\documentclass[preprint,12]{elsarticle}
\usepackage{fullpage}
\usepackage[usenames]{color}
\usepackage[colorlinks = true]{hyperref}
\usepackage[english]{babel}
\usepackage{bm,bbm,amsmath,amssymb,amsthm,graphicx}
\usepackage{url}
\usepackage[caption=false,font=normalsize,labelfont=sf,textfont=sf]{subfig}
\usepackage[title]{appendix}
\usepackage{algorithm}
\usepackage{algpseudocode} 
\usepackage{todonotes} 
\usepackage{tikz}
\usetikzlibrary{calc,patterns,decorations.pathmorphing,decorations.markings}
\usepackage[normalem]{ulem}

\numberwithin{equation}{section}

\newcommand{\MAT}{\left[ \begin{array}}  
\newcommand{\mat}{\end{array} \right]}
\newtheorem{Definition}{Definition}[section]

\newtheorem{Lemma}{Lemma}[section]
\newtheorem{Theorem}{Theorem}[section]

\newtheorem{Q}{Question}[section]

\def \minimize {\operatorname*{minimize}}

\def \st {\operatorname*{subject\ to\ }}

\def \A {\mathbf{A}}
\def \AA {\mathcal{A}}

\def \EEE{\mathbb{E}}

\def \HH {\mathcal{H}}

\def \NN {\mathcal{N}}

\def \PP {\mathcal{P}}

\def \r{\bm{r}}

\def \RRR {\mathbb{R}}
\def \S {\mathbf{S}}

\def \U {\mathbf{U}}

\def \W {\mathbf{W}}

\def \x {\bm{x}}
\def \xs {\bm{x}^\star}

\def \X {\mathbf{X}}
\def \Xt {\widetilde{\mathbf{X}}}
\def \Xh {\widehat{\mathbf{X}}}
\def \Xs {\mathbf{X}^\star}

\def \y {\bm{y}}
\def \Y {\mathbf{Y}}

\def \zero {\mathbf{0}}

\begin{document}
\begin{frontmatter}

\title{Stochastic Iterative Hard Thresholding for Low-Tucker-Rank Tensor Recovery}

\author[1]{Rachel Grotheer}
\author[2]{Shuang Li}
\author[3]{Anna Ma}
\author[4]{ Deanna Needell}
\author[5]{Jing Qin}

\address[1]{Goucher College, rachel.grotheer@goucher.edu.}
\address[2]{Colorado School of Mines, shuangli@mines.edu.}
\address[3]{University of California, Irvine, anna.ma@uci.edu.}
\address[4]{University of California, Los Angeles, deanna@math.ucla.edu.}
\address[5]{University of Kentucky, jing.qin@uky.edu.}





\begin{abstract}

Low-rank tensor recovery problems have been widely studied in many applications of signal processing and machine learning. Tucker decomposition is known as one of the most popular decompositions in the tensor framework. In recent years, researchers have developed many state-of-the-art algorithms to address the problem of low-Tucker-rank tensor recovery. Motivated by the favorable properties of the stochastic algorithms, such as stochastic gradient descent and stochastic iterative hard thresholding, we aim to extend the well-known stochastic iterative hard thresholding algorithm to the tensor framework in order to address the problem of recovering a low-Tucker-rank tensor from its linear measurements. We have also developed linear convergence analysis for the proposed method and conducted a series of experiments with both synthetic and real data to illustrate the performance of the proposed method.


\end{abstract}


\end{frontmatter}

\section{Introduction}
\label{intr}

Tensors are high-dimensional extensions of vectors and matrices. 
There are many different kinds of tensor decompositions, among which, the Canonical Polyadic (CP) decomposition and Tucker decomposition are the most popular~\cite{sokolnikoff1951tensor,zhang2017matrix}. In recent years, low-rank tensor recovery problems have gained a great amount of attention in various applications including hyperspectral image restoration~\cite{fan2017hyperspectral}, video processing~\cite{bengua2017efficient}, signal processing~\cite{li2015overcomplete, li2017jazz}, and simultaneous blind deconvolution and phase retrieval~\cite{li2019simultaneous}. Unlike low-rank matrix recovery problems, which often use nuclear norm minimization as a popular heuristic for rank minimization, the computation of the nuclear norm for high order tensors is NP-hard~\cite{hillar2013most, friedland2018nuclear}.

Over the decades, the iterative hard thresholding (IHT) algorithm has been widely used in 
compressive sensing~\cite{blumensath2009iterative,blanchard2015cgiht,carrillo2013lorentzian} and low-rank matrix recovery~\cite{tanner2013normalized,chunikhina2014performance,geng2015accelerated}. It has many extensions, such as the stochastic variant proposed in~\cite{nguyen2017linear}, which was further extended to the multiple measurement vector framework in~\cite{qin2017stochastic}. Inspired by the idea of using the IHT algorithm in low-rank matrix recovery problems, the authors in~\cite{rauhut2017low} extended the IHT algorithm to the tensor framework and proposed the Tensor IHT (TIHT) algorithm as an alternative to the tensor nuclear norm minimization. The authors of~\cite{li2019simultaneous} then combined this TIHT algorithm with higher-order singular value decomposition (HOSVD), a type of Tucker decomposition, to solve a low-rank tensor recovery problem formulated from a simultaneous blind deconvolution and phase retrieval problem. Another recent work~\cite{de2017low} also extends the IHT algorithm to the problem of low-rank tensor recovery based on a low-Tucker-rank approximation technique named sequentially optimal modal projections.   

The stochastic versions of gradient descent algorithms and IHT algorithms usually have many favorable properties. For example, these algorithms do not need to compute the full gradient, which makes it possible for them to be utilized in large scale problems where computing the full gradient is very expensive.   
These properties inspired us to extend the stochastic IHT
algorithm to the tensor framework and introduce the Stochastic Tensor IHT (StoTIHT) algorithm to recover a low-Tucker-rank tensor from its linear measurements. In this work, we provide convergence analysis for the proposed StoTIHT algorithm, based on a Tucker decomposition of the tensor, under the assumption that the linear operator used to obtain the measurements satisfies a tensor restricted isometry property (TRIP). Our simulations also indicate that the proposed StoTIHT algorithm converges much faster than the original TIHT algorithm in a large scale setting. 


The remainder of this work is organized as follows. In Section~\ref{sec:prel}, we briefly review some fundamental concepts and definitions used in the tensor framework. We formulate our low-rank tensor recovery problem in Section~\ref{sec:prob} and present the proposed StoTIHT algorithm in Section~\ref{sec:StoTIHT}. We then introduce the linear convergence analysis for our proposed StoTIHT algorithm in Section~\ref{sec:linCon} and illustrate its performance with both synthetic and real data in Section~\ref{sec:simu}. Finally, we conclude our work in Section~\ref{sec:conc}.

\section{Preliminaries}
\label{sec:prel}


In this section, we briefly review some fundamental concepts and definitions used in the tensor framework~\cite{sokolnikoff1951tensor,zhang2017matrix,kolda2009tensor}. We denote a $d$-th order tensor as $\X \in \RRR^{n_1\times n_2\times\cdots\times n_d}$. Vectors and matrices can be viewed as low-dimensional tensors with $d=1$ and 2, respectively.
Denote $\X^{\{i\}}\in \RRR^{n_i \times (n_1n_2\cdots n_{i-1}n_{i+1} \cdots n_d)}$ as the {\em mode-$i$ matricization} or the {\em $i$-th unfolding} of a tensor $\X\in \RRR^{n_1\times n_2\times\cdots\times n_d}$\footnote{One can refer to~\cite{kolda2009tensor} for a more detailed definition and some easily understandable examples.}. 
Similar to the matrix case, it is possible to vectorize a tensor, resulting in a column vector containing all of the elements of the tensor. Note that the ordering of the elements is not uniquely defined. In this work, the ordering is consistent. In particular,
denote $\text{vec}(\cdot)$ 
as an operator used to vectorize a matrix or a tensor. For a tensor $\X$, we choose the vectorization as $\text{vec}(\X) = \text{vec}(\X^{\{1\}})$. That is, we chose to vectorize the mode-1 matricization of the tensor $\X$ given by $\X^{n_1 \times (n_2n_3\cdots n_d)}$.
The {\em inner product} of two tensors $\X_1,\X_2 \in \RRR^{n_1\times n_2\times\cdots\times n_d}$ is then defined as 
\begin{align*}
\langle \X_1,\X_2 \rangle \triangleq \text{vec}(\X_2)^\top \text{vec}(\X_1).	
\end{align*}
The induced {\em Frobenius norm} is then defined as
\begin{align*}
\|\X\|_F \triangleq \sqrt{\langle \X,\X \rangle}.	
\end{align*}

The {\em Tucker rank} of tensor $\X$ is then defined as a tuple $\r=(r_1,\cdots,r_d)$ with $r_i = \text{rank}(\X^{\{i\}})$. The Tucker decomposition is one of the most popular tensor decompositions and one can find more details in~\cite{tucker1963implications,tucker1964extension}. The family of Tucker decompositions include the HOSVD of a tensor $\X\in \RRR^{n_1\times n_2\times\cdots\times n_d}$, which is given as
\begin{align}
\X = \S \times_1 \U^{(1)} \cdots \times_d \U^{(d)}.	
\label{eqn:tucker_decomp}
\end{align}
Here, $\S\in\RRR^{r_1\times\cdots\times r_d}$ and $\U^{(i)}\in\RRR^{n_i \times r_i}$ denote the core tensor and the basis, respectively. One can refer to~\cite{sokolnikoff1951tensor,zhang2017matrix} for more details about the properties of the core tensor and basis. The product $\times_i$ is the {\em mode-$i$ (matrix) product} of the tensor, that is, the product of a tensor and a matrix along the $i$-th mode of the tensor.

\section{Problem Formulation}
\label{sec:prob}

In this work, we consider the recovery of a rank-$\r$ tensor $\Xs \in \RRR^{n_1\times n_2\times\cdots\times n_d}$ from its linear measurements $\y = \AA(\Xs) \in \RRR^{m}$, where $\AA: \RRR^{n_1\times n_2\times\cdots\times n_d} \rightarrow \RRR^m$ is a linear operator used to obtain the measurements. In particular, the $i$-th element of $\y$ is given as 
\begin{align}
\y(i) = \AA_i(\Xs) = \langle \A_i,\Xs \rangle,	\quad i=1,\ldots,m,
\label{eqn:meas_model}
\end{align}
where $\A_i \in \RRR^{n_1\times n_2\times\ldots\times n_d}$ is a sensing tensor. We observe that the cost function $F(\X)$ can be rewritten as
\begin{equation}
\begin{aligned}
F(\X) & \triangleq \frac{1}{2m} \|\y-\AA(\X)\|_2^2 = \frac{1}{2m} \sum_{i=1}^m (\y(i)-\langle \A_i,\X \rangle)^2	\\
& = \frac 1 M \sum_{i=1}^M \left( \frac{1}{2b} \sum_{j=(i-1)b+1}^{ib} ( \y(j)-\langle \A_j,\X \rangle )^2 \right)\\
& = \frac 1 M \sum_{i=1}^M \frac{1}{2b} \| \y_{b_i}-\AA_{b_i}(\X) \|_2^2
 \triangleq \frac 1 M \sum_{i=1}^M f_i(\X),
\label{eqn:cost}
\end{aligned}
\end{equation}
where we decompose the measurement vector $\y\in\RRR^m$ into $M$ non-overlapping vectors $\y_{b_i}\in\RRR^{b}$, $i=1,\ldots,M$. 
Note that we can choose $b$ to be an integer and let $M = \lceil m/b\rceil$. We denote $\AA_{b_i}: \RRR^{n_1\times n_2\times\cdots\times n_d} \rightarrow \RRR^b$ as a linear operator with the $j$-th entry of $\AA_{b_i}(\X)$ being $\langle \A_{(i-1)b+j},\X \rangle$, $j=1,\ldots b$. It can be seen that each function $f_i(\X)$ is associated with a collection of measurements $\y_{b_i}$.  

Due to the low-rankness of $\Xs$, a standard approach to recover $\Xs$ is to solve the following minimization 
\begin{align}
\minimize_{\X\in \RRR^{n_1\times n_2\times\ldots\times n_d}}~F(\X) \quad \st~\text{rank}(\X)\leq \r, 	\label{eqn:mainopt}
\end{align}
where the cost function $F(\X)$ is defined in~\eqref{eqn:cost}. The above optimization has been heavily studied in existing literature. The authors in~\cite{gandy2011tensor, mu2014square} relax the above optimization by minimizing the sum of the nuclear norm of the tensor matricizations. However, this kind of relaxation is actually non-optimal~\cite{recht2010guaranteed}. Inspired by the idea of using the IHT algorithm for compressive sensing and low-rank matrix recovery problems, the authors in~\cite{rauhut2017low} extend the IHT algorithm~\cite{blumensath2009iterative} to the tensor framework and propose the TIHT algorithm as an alternative to the tensor nuclear norm minimization. In particular, to recover $\Xs$, the TIHT algorithm consists of the following two steps at iteration $t$:
\begin{align}
\widetilde{\X}^t &= \X^t + \mu \AA^*(\y-\AA(\X^t)), \label{eqn:tiht1} \\
\X^{t+1} &= \mathcal{H}_{\bm{r}}(\widetilde{\X}^t). \label{eqn:tiht2}
\end{align}
Here, $\mu$ is the stepsize and $\AA^*:\RRR^m \rightarrow \RRR^{n_1\times n_2\times\ldots\times n_d}$ is the adjoint operator of $\AA$. That is for tensor $\X$ and vector $\y$, $\langle \AA (\X), \y \rangle = \langle \X, \AA^*(\y)\rangle$.
$\mathcal{H}_{\bm{r}}(\X)$ computes a rank-$\bm{r}$ approximation of a tensor $\X$ with HOSVD. Note that the second step~\eqref{eqn:tiht2} is not straightforward, and the authors in~\cite{rauhut2017low} required that the assumption
\begin{align}
\|\HH_{\r}(\Xt^t)-\Xt^t\|_F \leq \eta \|\Xt^t_{\text{best}}-\Xt^t\|_F
\label{eqn:assmp_tiht}	
\end{align}
held for all $t=1,2,\ldots,T$ with some $\eta\in[1,\infty)$. $\Xt^t_{\text{best}}$ is the best rank-$\r$ approximation of $\Xt^t$ with respect to the Tucker decomposition (given by the HOSVD), namely, $\Xt^t_{\text{best}} = \arg\min_{\text{rank}(\X)\leq \r} \|\Xt^t-\X\|_F$. We will also assume such an approximation exists in our theorem.




\section{The Proposed StoTIHT algorithm}
\label{sec:StoTIHT}



We observe that the linear measurements $\y$ can be rewritten as 
\begin{align*}
\y = \A\xs,
\end{align*}
where $\A \in \RRR^{m \times n_1 n_2 \cdots n_d}$ is a matrix with the $i$-th row being the vectorized version of $\A_i$, and $\xs$ is the vectorized version of $\Xs$. Then, we can update $\widetilde{\X}^t$ in~\eqref{eqn:tiht1} with
\begin{align}
\widetilde{\x}^t &= \x^t + \mu \A^\top(\y-\A\x^t),\label{eqn:tiht11}	
\end{align}
where $\widetilde{\x}$ is the vectorized version of $\widetilde{\X}$.

As previously stated, the stochastic variant of algorithms does not require computation of the full gradient and, thus can be much more efficient in large scale settings especially when the computation and/or storage of the full gradient is very expensive. Thus, we propose a stochastic variant of the TIHT algorithm (StoTIHT) by replacing \eqref{eqn:tiht11} with
\begin{align}
\widetilde{\x}^t &= \x^t + \frac{\mu}{Mp(i_t)} \A(i_t,:)^\top(\y_{b_{i_t}}-\A(i_t,:)\x^t),\label{eqn:tiht111}	
\end{align}
where $i_t$ is an index randomly selected from $[M]=\{1,2,\cdots,M\}$ with probability $p(i_t)$, and $\A(i_t,:)\in\RRR^{b \times n_1 n_2 \cdots n_d}$ denotes the $i_t$-th block of $\A$. This updating step is equivalent to 
\begin{equation}
\begin{aligned}
\Xt^t &= \X^t + \frac{\mu}{M p(i_t)} \frac{1}{b} \sum_{j=(i_t-1)b+1}^{i_t b}  \A_j(\y(j)-\langle \A_j,\X^t \rangle)\\
& = \X^t - \frac{\mu}{Mp(i_t)} \nabla f_{i_t}(\X^t)\label{eqn:tiht1111}	
\end{aligned}
\end{equation}
with $f_{i_t}(\X^t) \triangleq \frac{1}{2b} \sum_{j=(i_t-1)b+1}^{i_t b} ( \y(j)-\langle \A_j,\X \rangle )^2$ as in~\eqref{eqn:cost}. Based on the above analysis, we summarize the algorithm below.

\begin{algorithm}[H]
\caption{Stochastic Tensor Iterative Hard Thresholding (StoTIHT)}\label{alg:StoTIHT}
\begin{algorithmic}[1]
\State\textbf{Input:} $\r$, $\mu$, and $p(i)$.
\State\textbf{Output:} $\Xh=\X^{T}$.
\State\textbf{Initialize:} $\X^0=\zero$ 
\For {$t=0,1,\ldots,T-1$}
\State Randomly select a batch index $i_t\in [M]$ with probability $p(i_t)$
\State Compute the gradient $\nabla f_{i_t}(\X^t)$ as given in~\eqref{eqn:tiht1111}
\State $\Xt^t = \X^t - \frac{\mu}{Mp(i_t)} \nabla f_{i_t}(\X^t)$
\State $\X^{t+1} = \mathcal{H}_{\bm{r}}(\widetilde{\X}^t)$
\State If the stopping criteria are met, exit.
\EndFor
\end{algorithmic}
\end{algorithm}


\section{Linear Convergence for StoTIHT}
\label{sec:linCon}

We present a linear convergence analysis for the proposed StoTIHT algorithm in this section. We first introduce the tensor restricted isometry property (TRIP) in the following definition.


\begin{Definition}\label{def:TRIP}
(TRIP)~\cite{rauhut2017low} 
Let $\AA: \RRR^{n_1\times n_2\times\cdots\times n_d} \rightarrow \RRR^m$ and $\AA_{b_i}: \RRR^{n_1\times n_2\times\cdots\times n_d} \rightarrow \RRR^b$ be the two linear operators defined in Section~\ref{sec:prob}. For a fixed tensor Tucker decomposition and a corresponding Tucker rank $\r$, we say $\AA$ and $\AA_{b_i}$ satisfy the TRIP if there exists a tensor restricted isometry constant $\delta_{\r}$ such that 
\begin{align}
\frac 1 m \|\AA(\X)\|_2^2 &\geq (1-\delta_{\r})\|\X\|_F^2 \label{eqn:trip1}\\
\frac 1 b \|\AA_{b_i}(\X)\|_2^2 &\leq (1+\delta_{\r})\|\X\|_F^2	\label{eqn:trip2}
\end{align}
hold for all tensors $\X\in \RRR^{n_1\times n_2\times\cdots\times n_d}$ of Tucker-rank at most $\r$.	
\end{Definition}

Note that~\eqref{eqn:trip2} is stronger than $\frac 1 m \|\AA(\X)\|_2^2 \leq (1+\delta_{\r})\|\X\|_F^2$, but we will need~\eqref{eqn:trip2} in the proof of the main theorem. We call the linear operator $\AA$ defined in Section~\ref{sec:prob} a random Gaussian linear operator if the entries of all of the sensing tensors $\A_i$ are random Gaussian variables. It is shown in~\cite{rauhut2017low} that the random Gaussian linear operators $\AA$ and $\AA_{b_i}$ satisfy the TRIP with high probability as long as $m\geq C\delta_{\r}^{-2}(r^d+dnr)$ and $b\geq C\delta_{\r}^{-2}(r^d+dnr)$\footnote{These $C$'s are not necessarily the same constant, according to~\cite{rauhut2017low}.} with $n = \max \{n_i, i\in[d] \}$ and $r = \max\{r_t: t\in T_I \}$. 
Here, $T_I$ is corresponding partition tree of the vertices $\alpha \subset [d] \triangleq \{1,\cdots,d\}$. See~\cite{rauhut2017low} for a more detailed description of the construction of $T_I$.


Now, we are in the position to state our main results in the following theorem.

\begin{Theorem}\label{thm:convergence}
Assume that the operators $\AA: \RRR^{n_1\times n_2\times\cdots\times n_d} \rightarrow \RRR^m$ and $\AA_{b_i}: \RRR^{n_1\times n_2\times\cdots\times n_d} \rightarrow \RRR^b$ used to obtain the linear measurements $\y$ satisfy the TRIP defined in Definition~\ref{def:TRIP}. 
Let $\Xs$ be a feasible solution of the optimization problem~\eqref{eqn:mainopt}. Denote $\X^0$ as the initial tensor. We also assume that the Tucker-rank-$\r$ approximation operator $\HH_{\r}(\cdot)$ satisfies~\eqref{eqn:assmp_tiht} for all $t=0,1,\ldots,T-1$ with some $\eta\in[1,\infty)$. Using Algorithm~\ref{alg:StoTIHT}, one can guarantee that the expectation of the recovery error is bounded by
\begin{align*}
\EEE_{I_t} \|\X^{t+1}-\Xs\|_F \leq \kappa^{t+1} 	\|\X^0-\Xs \|_F + \sigma_{\Xs}	
\end{align*}
at the $t$-th iteration. $I_t = \{i_1,i_2,\ldots,i_t\}$ denotes the set containing all indices $i_1,i_2,\ldots,i_t$ randomly selected at and before iteration $t$.  Here, $\kappa$ and $\sigma_{\Xs}$ are the contraction coefficient and tolerance parameter, which are defined as 
\begin{align*}
\kappa &\triangleq 2\sqrt{1-(2-\mu\alpha_{3\r})\mu \rho_{3\r}^-}+ \sqrt{\eta^2-1} \sqrt{1+\mu^2 \alpha_{3\r} \rho_{3\r}^+ - 2\mu \rho_{3\r}^-}	\\
\sigma_{\Xs} & \triangleq \frac{\mu}{M\min_{i\in[M]} p(i)}\left(2\EEE_{i_t}\left\|  \PP_{U^t}(\nabla f_{i_t}(\Xs)) \right\|_F+ \sqrt{\eta^2-1}\EEE_{i_t }\left\|\nabla f_{i_t}(\Xs) \right\|_F \right)
\end{align*}
with
$ \rho_{\r}^+ \triangleq 2(1+\delta_{\r}), ~ \rho_{\r}^- \triangleq 1-\delta_{\r},~
\alpha_{\r} \triangleq \max_i \frac{\rho_{\r}^+}{Mp(i)},
$
and $i_t$ being an index selected from $[M]$ with probability $p(i_t)$. $U^t$ is defined as a subspace of $\RRR^{n_1\times n_2\times\ldots\times n_d}$ spanned by $\Xs,~\X^t$, and $\X^{t+1}$. $\PP_{U^t}:\RRR^{n_1\times n_2\times\ldots\times n_d} \rightarrow U^t$ is then defined as the orthogonal projection onto $U^t$.
\end{Theorem}


The above theorem indicates that the proposed StoTIHT algorithm is guaranteed to linearly converge to the true solution even when the full gradient computation is not available. 
Therefore, the proposed StoTIHT algorithm has a significant computational advantage in large-scale settings where the number of measurements $m$ is very large and the computation of the full gradient can be very expensive, since it only needs to compute a partial gradient.

We adapt the proofs of StoIHT in~\cite{nguyen2017linear} and TIHT in~\cite{rauhut2017low} to the show the linear convergence of our proposed StoTIHT algorithm. Note that in what follows ``rank" denotes Tucker rank. We first present a key lemma that will be used in the proof of Theorem~\ref{thm:convergence}.

\begin{Lemma} \label{lem:key1}
Denote $i$ as the index randomly selected from $[M]$ with probability $p(i)$. For any fixed low-rank tensors $\X'$ and $\X$, let $U$ be a space that contains the space spanned by $\X'$ and $\X$. Assume that the rank of any tensor in $U$ is at most $\r$. Then, we have 	
\begin{align}
&\EEE_{i}\left\| \X' -\X - \frac{\mu}{Mp(i)} \PP_{U}(\nabla f_{i}(\X')-\nabla f_{i}(\X))  \right\|_F \leq \sqrt{1-(2-\mu\alpha_{\r})\mu \rho_{\r}^-}\|\X' -\X \|_F, \label{eqn:keylem11}	\\
&\EEE_{i}\left\| \X' -\X - \frac{\mu}{Mp(i)}(\nabla f_{i}(\X')-\nabla f_{i}(\X))  \right\|_F \leq \sqrt{1+\mu^2 \alpha_{\r} \rho_{\r}^+ - 2\mu \rho_{\r}^-} \|\X' -\X \|_F. \label{eqn:keylem12}
\end{align}
\end{Lemma}
To prove Lemma~\ref{lem:key1}, we need the following two lemmas.

\begin{Lemma}\label{lem:key1_2}
For any two low-rank tensors $\X'$ and $\X$, let $U$ be a space that contains the space spanned by $\X'$ and $\X$. Assume that the rank of any tensor in $U$ is at most $\r$. Suppose that the two linear operators $\AA$ and $\AA_{b_i}$ satisfy the TRIP defined in Definition~\ref{def:TRIP}. Then, the functions $F(\X)$ and $f_i(\X)$	defined in~\eqref{eqn:cost} satisfy
\begin{align}
&\rho_{\r}^-\|\X'-\X\|_F^2 \leq \langle \X'-\X, \nabla F(\X') - \nabla F(\X) \rangle, \label{eqn:key1_2_1}\\
&\|\nabla f_i(\X') - \nabla f_i(\X)\|_F \leq \rho_{\r}^+\|\X'-\X\|_F \label{eqn:key1_2_2}
\end{align}
for all low-rank tensors $\X'$ and $\X$. 

\end{Lemma}

\begin{proof}
Recall that 
\begin{align*}
F(\X) & = \frac {1}{2m} \sum_{i=1}^m (\y(i)-\langle \A_i,\X \rangle)^2,\\
f_i(\X) & = \frac{1}{2b} \sum_{j=(i-1)b+1}^{ib} ( \y(j)-\langle \A_j,\X \rangle )^2.	
\end{align*}
With some fundamental calculations, we obtain their gradients
\begin{align*}
\nabla F(\X) &= \frac 1 m \sum_{i=1}^m \A_i(\langle \A_i,\X \rangle-\y(i)),\\
\nabla f_i(\X) & = \frac 1 b \sum_{j=(i-1)b+1}^{ib} \A_j(\langle \A_j,\X \rangle-\y(j)). 	 
\end{align*}
 
It follows from the TRIP~\eqref{eqn:trip1} that
\begin{align*}
\langle \X'-\X, \nabla F(\X') - \nabla F(\X) \rangle &=\frac 1 m\left\langle \X'-\X,\sum_{i=1}^m \A_i \langle \A_i,\X'-\X \rangle   \right\rangle \\
& = \frac 1 m \sum_{i=1}^m \langle \A_i,\X'-\X \rangle^2\\
& = \frac 1 m \|\AA(\X'-\X)\|_2^2\\
&\geq (1-\delta_{\r})\|\X'-\X\|_F^2.
\end{align*}	
Thus, we finish the proof of~\eqref{eqn:key1_2_1} by setting $\rho_{\r}^- = 1-\delta_{\r}$.

Note that
\begin{align*}
\frac 1 b \|\AA_{b_i}(\X'-\X)\|_2^2 &= \frac 1 b \sum_{j=(i-1)b+1}^{ib} \langle \A_j,\X'-\X \rangle^2 = \frac 1 b \left\langle  \sum_{j=(i-1)b+1}^{ib} \A_j \langle \A_j,\X'-\X \rangle,\X'-\X  \right\rangle\\
& = \langle \nabla f_i(\X') - \nabla f_i (\X),\X'-\X \rangle\\
& = (f_i(\X')-f_i(\X)- \langle \nabla f_i(\X),\X'-\X \rangle )+(f_i(\X)-f_i(\X')- \langle \nabla f_i(\X'),\X-\X' \rangle ).
\end{align*}
Without loss of generality, we assume that 
\begin{align*}
f_i(\X')-f_i(\X)- \langle \nabla f_i(\X),\X'-\X \rangle \leq f_i(\X)-f_i(\X')- \langle \nabla f_i(\X'),\X-\X' \rangle. 	
\end{align*}
Then, it follows from the TRIP~\eqref{eqn:trip2} that
\begin{align}
f_i(\X')-f_i(\X)- \langle \nabla f_i(\X),\X'-\X \rangle \leq (1+\delta_{\r}) \|\X'-\X\|_F^2.
\label{eqn:zhou2}
\end{align}
Note that $f_i(\X)$ is a convex function with respect to $\X$, together with Lemma 4 in~\cite{zhou2018fenchel}, we then get~\eqref{eqn:key1_2_2} by setting $\rho_{\r}^+ = 2(1+\delta_{\r})$.

\end{proof}

\begin{Lemma}\label{lem:key1_1}
(Co-coercivity) For any two low-rank tensors $\X'$ and $\X$, let $U$ be a space that contains the space spanned by $\X'$ and $\X$. Assume that the rank of any tensor in $U$ is at most $\r$. Then, we have
\begin{align}
\|\PP_U(\nabla f_i(\X')-\nabla f_i(\X))\|_F^2 \leq \rho_{\r}^+ \langle \X'-\X, \nabla f_i(\X') - \nabla f_i(\X) \rangle	
\label{eqn:coco}
\end{align}
	
\end{Lemma}

\begin{proof}

Exchanging the role of $\X'$ and $\X$ in~\eqref{eqn:zhou2}, we get
\begin{align*}
f_i(\X) - 	f_i(\X') - \langle \nabla f_i(\X'),\X-\X'	\rangle \leq \frac 1 2\rho_{\r}^+ \|\X'-\X\|_F^2.
\end{align*}
Adding the above inequality and~\eqref{eqn:zhou2} together, we obtain
\begin{align*}
\langle \X'-\X,\nabla f_i(\X')-	\nabla f_i(\X)\rangle \leq \rho_{\r}^+ \|\X'-\X\|_F^2.
\end{align*}

Define a function 
\begin{align*}
G_i(\W) \triangleq f_i(\W) - 	\langle \nabla f_i(\X),\W	\rangle.
\end{align*}
Then, we have
\begin{align*}
\|\nabla G_i(\W_1) - \nabla G_i(\W_2)\|_F = \| \nabla f_i(\W_1) - \nabla f_i(\W_2) \|_F  \leq \rho_{\r}^+\|\W_1-\W_2\|_F
\end{align*}
holds for any $\W_1$ and $\W_2$ with their span belonging to $U$. This implies that we can get a similar inequality as in~\eqref{eqn:zhou2}, namely
\begin{align}
G_i(\W_1) - G_i(\W_2) - \langle \nabla G_i(\W_2),\W_1-\W_2\rangle \leq \frac 1 2\rho_{\r}^+ \|\W_1-\W_2\|_F^2.
\label{eqn:drss_G1}
\end{align}

Note that 
\begin{equation}
\begin{aligned}
G_i(\W) - G_i(\X) &= f_i(\W) - f_i(\X) - 	\langle \nabla f_i(\X),\W - \X\rangle	\\
& = \frac{1}{2b} \sum_{j=(i-1)b+1}^{ib} \langle \A_j,\W-\X\rangle^2 \geq 0  
\label{eqn:GWmX}
\end{aligned}
\end{equation}
holds for all $\W\in U$. Here, the second equality follows by plugging the expression of $f_i(\W),~f_i(\X)$, and $\nabla f_i(\X)$. Define $\W \triangleq \X' - \frac{1}{\rho_{\r}^+} \PP_U \nabla G_i(\X')$, it can be seen that $\W \in U$ since both $\X'$ and $\PP_U \nabla G_i(\X')$ belong to $U$. Then, applying~\eqref{eqn:drss_G1} and \eqref{eqn:GWmX}, we have
\begin{align*}
G_i(\X) &\leq G_i(\W) =  G_i \left( \X' - \frac{1}{\rho_{\r}^+} \PP_U \nabla G_i(\X') \right)	\\
&\leq G_i(\X') + \left\langle \nabla G_i(\X'), - \frac{1}{\rho_{\r}^+} \PP_U \nabla G_i(\X')\right\rangle + \frac{1}{2\rho_{\r}^+} \|  \PP_U \nabla G_i(\X')\|_F^2\\
&= G_i(\X')-\frac{1}{2\rho_{\r}^+} \|  \PP_U \nabla G_i(\X')\|_F^2.
\end{align*}
Plugging the definition of $G_i(\X)$ and $G_i(\X')$ into the above inequality gives
\begin{align*}
\frac{1}{2\rho_{\r}^+} \|  \PP_U \nabla G_i(\X')\|_F^2 
&= \frac{1}{2\rho_{\r}^+} \|  \PP_U ( \nabla f_i(\X') - \nabla f_i(\X) )\|_F^2  \\
&\leq G_i(\X') - G_i(\X) = 	f_i(\X')-f_i(\X) - \langle \nabla f_i(\X),\X'-\X\rangle.
\end{align*}
Finally, we can obtain~\eqref{eqn:coco} by summing the two inequalities with $\X'$ and $\X$ exchanged.  

\end{proof}

\noindent
{\em Proof of Lemma~\ref{lem:key1}.} With the above two lemmas, we obtain 
\begin{align*}
&\EEE_{i}\left\| \X' -\X - \frac{\mu}{Mp(i)} \PP_{U}(\nabla f_{i}(\X')-\nabla f_{i}(\X))  \right\|_F^2 \\
= & \|\X'-\X\|_F^2  + \EEE_i \frac{\mu^2}{(Mp(i))^2} \| \PP_{U}(\nabla f_{i}(\X')-\nabla f_{i}(\X)) \|_F^2 - 2\mu \EEE_i \left\langle \X'-\X,\frac{1}{Mp(i)}\PP_{U}(\nabla f_{i}(\X')-\nabla f_{i}(\X))  \right\rangle\\
\leq & \|\X'-\X\|_F^2  + \mu^2\EEE_i \frac{\rho_{\r}^+}{(Mp(i))^2}  \langle \X'-\X, \nabla f_i(\X') - \nabla f_i(\X) \rangle	- 2\mu \EEE_i \left\langle \X'-\X,\frac{1}{Mp(i)}(\nabla f_{i}(\X')-\nabla f_{i}(\X))  \right\rangle\\
\leq & \|\X'-\X\|_F^2  +  \left(\mu^2 \max_i \frac{\rho_{\r}^+}{Mp(i)}-2\mu\right) \EEE_i   \left\langle \X'-\X, \frac{1}{Mp(i)} (\nabla f_i(\X') - \nabla f_i(\X)) \right\rangle	\\
= & \|\X'-\X\|_F^2  -  \left(2\mu - \mu^2 \alpha_{\r}\right) \left\langle \X'-\X, \nabla F(\X') - \nabla F(\X)\right\rangle\\
\leq & \|\X'-\X\|_F^2 -  \left(2\mu -\mu^2 \alpha_{\r}\right)\rho_{\r}^-\|\X'-\X\|_F^2.  
\end{align*}
Here, the first inequality follows from~\eqref{eqn:coco} and the definition of $U$. The second equality follows from $\EEE_i \frac{1}{Mp(i)} \nabla f_i(\X) = \nabla F(\X)$ and $\alpha_{\r} \triangleq \max_i \frac{\rho_{\r}^+}{Mp(i)}$. The last inequality follows from~\eqref{eqn:key1_2_1}. Then, we complete the proof of~\eqref{eqn:keylem11} by applying the well known inequality $(\EEE Z)^2 \leq \EEE Z^2$. 

Similar to the proof of~\eqref{eqn:keylem11}, we also have
\begin{align*}
&\EEE_{i}\left\| \X' -\X - \frac{\mu}{Mp(i)} (\nabla f_{i}(\X')-\nabla f_{i}(\X))  \right\|_F^2 \\
= & \|\X'-\X\|_F^2  + \EEE_i \frac{\mu^2}{(Mp(i))^2} \| \nabla f_{i}(\X')-\nabla f_{i}(\X) \|_F^2 - 2\mu \EEE_i \left\langle \X'-\X,\frac{1}{Mp(i)} (\nabla f_{i}(\X')-\nabla f_{i}(\X))  \right\rangle\\
= & \|\X'-\X\|_F^2  + \EEE_i \frac{\mu^2}{(Mp(i))^2} \| \nabla f_{i}(\X')-\nabla f_{i}(\X) \|_F^2 - 2\mu \left\langle \X'-\X,\nabla F(\X')-\nabla F(\X) \right\rangle\\  
\leq & \|\X'-\X\|_F^2 + \EEE_i \frac{\mu^2}{(Mp(i))^2}(\rho_{\r}^+)^2\|\X'-\X\|_F^2 - 2\mu \rho_{\r}^-\|\X'-\X\|_F^2\\
\leq & (1+\mu^2 \alpha_{\r} \rho_{\r}^+-2\mu \rho_{\r}^-)\|\X'-\X\|_F^2,
\end{align*}
where the first inequality follows from~\eqref{eqn:key1_2_1} and \eqref{eqn:key1_2_2}, and the last inequality follows from 
\begin{align*}
\EEE_i \frac{\mu^2}{(Mp(i))^2}(\rho_{\r}^+)^2\leq \mu^2 \max_i \frac{\rho_{\r}^+}{Mp(i)}	\EEE_i \frac{\rho_{\r}^+}{Mp(i)} = \mu^2 \alpha_{\r} \sum_{i=1}^M \frac{\rho_{\r}^+}{Mp(i)} p(i) = \mu^2 \alpha_{\r} \rho_{\r}^+.
\end{align*}
Therefore, by applying $(\EEE Z)^2 \leq \EEE Z^2$, we finish the proof of~\eqref{eqn:keylem12}. 

We are now prepared to prove Theorem~\ref{thm:convergence}.
\begin{proof}
Denote $\Xt^t_{\text{best}}$ as the best rank-$\r$ approximation of $\Xt^t$. Assume that the rank-$\r$ approximation operator $\HH_{\r}(\cdot)$ satisfies
\begin{align*}
\|\HH_{\r}(\Xt^t)-\Xt^t\|_F \leq \eta \|\Xt^t_{\text{best}}-\Xt^t\|_F	
\end{align*}
for all $t=1,2,\ldots$ with some $\eta\in[1,\infty)$. Note that for the HOSVD, \cite{grasedyck2010hierarchical} demonstrated a method to compute the HOSVD which obtains $\eta = \sqrt{2d - 3 }$, as well as methods that result in $\eta = \sqrt{2d - 2}$ and $\eta = (2 + \sqrt{2})\sqrt{d}$. We refer the reader to~\cite{grasedyck2010hierarchical} for a description of these methods.
Then, we have 
\begin{align*}
\|\X^{t+1}-\Xt^t\|_F = \|\HH_{\r}(\Xt^t)-\Xt^t\|_F \leq \eta \|\Xt^t_{\text{best}}-\Xt^t\|_F \leq \eta \|\Xt^t -\Xs\|_F	.
\end{align*}
It follows that
\begin{align*}
\eta^2 \|\Xt^t-\Xs\|_F^2 &\geq \|\X^{t+1}-\Xt^t\|_F^2	= \|\X^{t+1} -\Xs + \Xs-\Xt^t\|_F^2\\
& = \|\Xt^t -\Xs \|_F^2 + \|\X^{t+1} - \Xs\|_F^2 - 2 \langle \X^{t+1} - \Xs,\Xt^t -\Xs   \rangle.
\end{align*}
Simplifying, we get
\begin{align*}
\|\X^{t+1} - \Xs\|_F^2 &\leq 	2 \langle \X^{t+1} - \Xs,\Xt^t -\Xs   \rangle + (\eta^2-1)\|\Xt^t-\Xs\|_F^2. 
\end{align*}
Plugging in the updating step~\eqref{eqn:tiht1111} gives
\begin{align*}
\|\X^{t+1} - \X\|_F^2 
&\leq 2 \left\langle \X^{t+1} - \Xs,\X^t -\Xs - \frac{\mu}{Mp(i_t)} \nabla f_{i_t}(\X^t) \right\rangle + (\eta^2-1)\left\|\X^t -\Xs - \frac{\mu}{Mp(i_t)} \nabla f_{i_t}(\X^t)\right\|_F^2\\
& = 2 \left\langle \X^{t+1} - \Xs,\X^t -\Xs - \frac{\mu}{Mp(i_t)} (\nabla f_{i_t}(\X^t)-\nabla f_{i_t}(\Xs)) \right\rangle - 2 \left\langle \X^{t+1} - \Xs, \frac{\mu}{Mp(i_t)} \nabla f_{i_t}(\Xs) \right\rangle  \\
&~~~+ (\eta^2-1)\left\|\X^t -\Xs - \frac{\mu}{Mp(i_t)} \nabla f_{i_t}(\X^t)\right\|_F^2.
\end{align*}

Define $U^t$ as the subspace of $\RRR^{n_1\times n_2\times\ldots\times n_d}$ spanned by $\Xs,~\X^t$, and $\X^{t+1}$. Define $\PP_{U^t}:\RRR^{n_1\times n_2\times\ldots\times n_d} \rightarrow U^t$ as the orthogonal projection onto $U^t$. It can be seen that $\PP_{U^t}(\Xs) = \Xs$, $\PP_{U^t}(\X^t) = \X^t$, $\PP_{U^t}(\X^{t+1}) = \X^{t+1}$, and rank$(\PP_{U^t}(\Y)) \leq 3\bm{r}$ for all $\Y \in \RRR^{n_1\times n_2\times\ldots\times n_d}$. Then, we have
\begin{align*}
&\|\X^{t+1} - \X\|_F^2 \\
\leq & 2 \left\langle \X^{t+1} - \Xs,\X^t -\Xs - \frac{\mu}{Mp(i_t)} \PP_{U^t}(\nabla f_{i_t}(\X^t)-\nabla f_{i_t}(\Xs)) \right\rangle - 2 \left\langle \X^{t+1} - \Xs, \frac{\mu}{Mp(i_t)} \PP_{U^t}(\nabla f_{i_t}(\Xs)) \right\rangle  \\
&+ (\eta^2-1)\left\|\X^t -\Xs - \frac{\mu}{Mp(i_t)} \nabla f_{i_t}(\X^t)\right\|_F^2\\
\leq & 2\| \X^{t+1} - \Xs\|_F \left( \left\| \X^t -\Xs - \frac{\mu}{Mp(i_t)} \PP_{U^t}(\nabla f_{i_t}(\X^t)-\nabla f_{i_t}(\Xs))  \right\|_F  + \left\|  \frac{\mu}{Mp(i_t)} \PP_{U^t}(\nabla f_{i_t}(\Xs)) \right\|_F \right)\\
&+ (\eta^2-1)\left\|\X^t -\Xs - \frac{\mu}{Mp(i_t)} \nabla f_{i_t}(\X^t)\right\|_F^2.
\end{align*}

In general, it follows from $x^2-2ux-v \leq 0$ that $
x\leq u+\sqrt{u^2+v} \leq 2u+\sqrt{v}$, where the first inequality follows from completing the square and the second inequality from $\sqrt{a + b} \leq \sqrt{a} + \sqrt{b}$. Then, we can further obtain
\begin{align*}
\|\X^{t+1} - \Xs\|_F \leq &2 \left( \left\| \X^t -\Xs - \frac{\mu}{Mp(i_t)} \PP_{U^t}(\nabla f_{i_t}(\X^t)-\nabla f_{i_t}(\Xs))  \right\|_F  + \left\|  \frac{\mu}{Mp(i_t)} \PP_{U^t}(\nabla f_{i_t}(\Xs)) \right\|_F \right)\\
& + \sqrt{\eta^2-1}	\left\|\X^t -\Xs - \frac{\mu}{Mp(i_t)} \nabla f_{i_t}(\X^t)\right\|_F\\
\leq &2 \left( \left\| \X^t -\Xs - \frac{\mu}{Mp(i_t)} \PP_{U^t}(\nabla f_{i_t}(\X^t)-\nabla f_{i_t}(\Xs))  \right\|_F  + \left\|  \frac{\mu}{Mp(i_t)} \PP_{U^t}(\nabla f_{i_t}(\Xs)) \right\|_F \right)\\
& + \sqrt{\eta^2-1} \left(	\left\|\X^t -\Xs - \frac{\mu}{Mp(i_t)} (\nabla f_{i_t}(\X^t) - \nabla f_{i_t}(\Xs)) \right\|_F + \left\|\frac{\mu}{Mp(i_t)}\nabla f_{i_t}(\Xs) \right\|_F \right).
\end{align*}

Let $I_t$ be the set containing all indices $i_1,i_2,\ldots,i_t$ randomly selected at and before iteration $t$, i.e., $I_t = \{ i_1, i_2, \ldots, i_t \}.$ Denote the conditional expectation as $\EEE_{i_t | I_{t-1}} \| \X^{t+1}-\Xs \|_F \triangleq \EEE_{i_t}( \| \X^{t+1}-\Xs \|_F | I_{t-1})$. Note  that $\X^t$ can be viewed as a fixed tensor when conditioning on $I_{t-1}$. Taking conditional expectation on both sides of the above inequality gives
\begin{align*}
& \EEE_{i_t | I_{t-1}} \|\X^{t+1}-\Xs\|_F	\\
\leq &2 \left( \EEE_{i_t}\left\| \X^t -\Xs - \frac{\mu}{Mp(i_t)} \PP_{U^t}(\nabla f_{i_t}(\X^t)-\nabla f_{i_t}(\Xs))  \right\|_F  + \EEE_{i_t}\left\|  \frac{\mu}{Mp(i_t)} \PP_{U^t}(\nabla f_{i_t}(\Xs)) \right\|_F \right)\\
& + \sqrt{\eta^2-1} \left(	\EEE_{i_t}\left\|\X^t -\Xs - \frac{\mu}{Mp(i_t)} (\nabla f_{i_t}(\X^t) - \nabla f_{i_t}(\Xs)) \right\|_F + \EEE_{i_t }\left\|\frac{\mu}{Mp(i_t)}\nabla f_{i_t}(\Xs) \right\|_F \right)\\
\leq &2 \left(  \sqrt{1-(2-\mu\alpha_{3\r})\mu \rho_{3\r}^-}\|\X^t -\Xs \|_F  + \EEE_{i_t}\left\|  \frac{\mu}{Mp(i_t)} \PP_{U^t}(\nabla f_{i_t}(\Xs)) \right\|_F \right)\\
& + \sqrt{\eta^2-1} \left(	\sqrt{1+\mu^2 \alpha_{3\r} \rho_{3\r}^+ - 2\mu \rho_{3\r}^-} \|\X^t -\Xs \|_F   + \EEE_{i_t }\left\|\frac{\mu}{Mp(i_t)}\nabla f_{i_t}(\Xs) \right\|_F \right)\\
\leq & \left(2\sqrt{1-(2-\mu\alpha_{3\r})\mu \rho_{3\r}^-}+ \sqrt{\eta^2-1} \sqrt{1+\mu^2 \alpha_{3\r} \rho_{3\r}^+ - 2\mu \rho_{3\r}^-} \right)\|\X^t -\Xs \|_F\\
&+ \frac{\mu}{M\min_{i\in[M]} p(i)}\left(2\EEE_{i_t}\left\|  \PP_{U^t}(\nabla f_{i_t}(\Xs)) \right\|_F+ \sqrt{\eta^2-1}\EEE_{i_t }\left\|\nabla f_{i_t}(\Xs) \right\|_F \right)\\
=& \kappa \|\X^t -\Xs \|_F + \sigma_{\Xs},
\end{align*}
where the above equality follows from \eqref{eqn:keylem11} and \eqref{eqn:keylem12} in Lemma~\ref{lem:key1}.

By taking expectation on both sides of the above inequality with respect to $I_{t-1}$, we get
\begin{align*}
\EEE_{I_t} \|\X^{t+1}-\Xs\|_F \leq \kappa 	\EEE_{I_{t-1}}\|\X^t-\Xs \|_F + \sigma_{\Xs}, 
\end{align*}
which further implies that
\begin{align*}
\EEE_{I_t} \|\X^{t+1}-\Xs\|_F \leq \kappa^{t+1} 	\|\X^0-\Xs \|_F + \sigma_{\Xs}. 
\end{align*}
\end{proof}


\section{Numerical Simulations}
\label{sec:simu}


In this section, we illustrate the performance of the proposed StoTIHT algorithm by conducting a variety of experiments with both synthetic and real data. In what follows, we define the relative recovery error as $\frac{\|\Xs-\Xh\|_F}{\|\Xs\|_F}$ and consider the recovery as a success if the relative recovery error is less than $10^{-5}$. The experiments were conducted in Matlab R2014a installed on a laptop with an Intel(R) Core(TM) i7-4700MQ CPU @ 2.40GHz and 64G RAM.


In the first experiment, we work on third-order tensors (i.e., $d=3$) and set the parameters $n_1 = 5$, $n_2 = 5$, and $n_3 = 6$. The true Tucker-rank is $\r = (1,2,2)$.\footnote{Note that we choose a relatively small tensor size and rank in this experiment to reduce the problem size (and thus, computational time) while still demonstrating how the proposed algorithm performs.} 
With these parameters, we then generate the core tensor $\S\in\RRR^{r_1\times\cdots\times r_d}$ and the basis $\U^{(i)}\in\RRR^{n_i \times r_i}$ as random Gaussian tensors or matrices with entries following $\NN(0,1)$. Then, we create a low-rank tensor $\Xs$ as the target tensor according to the Tucker decomposition given in~\eqref{eqn:tucker_decomp}. We set $m=360$.
The sensing tensors $\A_i$ with $i=1,\ldots,m$ are also generated as random Gaussian tensors with entries satisfying $\NN(0,1)$, but followed by a normalization. In particular, each entry of the sensing tensors is rescaled by a factor of $1/\|\A\|_F$, where $\A \in \RRR^{m \times (n_1 n_2 \cdots n_d)}$ is a matrix with the $i$-th row being the vectorized version of $\A_i$. Then, we obtain the measurements $\y\in \RRR^{m}$ according to the measurement model~\eqref{eqn:meas_model}. For simplicity, we set $p(i) = \frac 1 M$ with $i = 1,\ldots,M$. Here, $M = \frac m b$ is the number of batches and we experiment with different batch sizes $b$, as shown in Figures~\ref{test_StoTIHT_epoch_m100} and~~\ref{test_StoTIHT_epoch_m100_time}. Note that in the case when $b=m$, the proposed StoTIHT algorithm reduces to the regular TIHT algorithm (black dashed line in Figures~\ref{test_StoTIHT_epoch_m100} and~\ref{test_StoTIHT_epoch_m100_time}). We set the stepsize $\mu=0.46m$ and use $\Xh$ to denote the recovered low-rank tensor. We present how the (a) cost function and (b) relative recovery error behave with respect to the number of epochs\footnote{An epoch is defined as the number of iterations needed to use $m$ rows. For deterministic algorithms like TIHT algorithm, an epoch is just one iteration, while for our StoTIHT algorithm, an epoch is $m/b$ iterations.} in Figure~\ref{test_StoTIHT_epoch_m100} and running time in Figure~\ref{test_StoTIHT_epoch_m100_time}. The presented results are all averaged over 100 trials. It can be seen that our proposed StoTIHT algorithm converges much faster than the original TIHT algorithm in this high dimensional setting.

Next, we fix the batch size as $b=0.5 m$ and repeat the above experiment with several values of $m$ and $\r$ as shown in Figure~\ref{test_StoTIHT_m_r}. We allow a maximum of 200 epochs for both the TIHT and StoTIHT algorithms in this experiment. The presented results are all averaged over 100 trials. We set the stepsize $\mu=m$. It can be seen that in this region with small $m$, the percentage of trials with successful recovery increases as we increase the number of measurements $m$ and decrease the Tucker rank. Then, we repeat this experiment with large number of measurements. We fix the batch size at $b=0.25m$ and set the maximum number of epochs at 80. We change the stepsize back to $\mu=0.5m$. As is shown in Figure~\ref{test_StoTIHT_m_r2}, in the region with a large $m$ and a small number of epochs, our proposed StoTIHT algorithm always successfully recovers the tensor and significantly outperforms the TIHT algorithm. This observation also coincides with Figure~\ref{test_StoTIHT_epoch_m100}, which indicates that the proposed StoTIHT algorithm converges much faster than the TIHT algorithm in the large scale setting. Therefore, we conjecture that the TIHT algorithm needs more epochs to get a successful recovery. To verify this conjecture, we conduct another experiment and present the number of epochs needed to achieve a successful recovery for different $r$ and $m$ in Figure~\ref{test_StoTIHT_m_r3}. We decrease the stepsize to $\mu=0.4m$. It can be seen that the proposed StoTIHT algorithm needs fewer epochs to get a perfect recovery especially when $m$ is large.

Finally, we test the proposed StoTIHT algorithm on real candle video data, which can be downloaded from the Dynamic Texture Toolbox in \url{http://www.vision.jhu.edu/code/}. To keep the recovery problem in a relatively small dimension, we truncate the video frames to be of size $30\times 30$ and only keep the first 10 frames. Then, the tensor to be recovered is of size $30\times 30\times 10$. We assume this tensor has Tucker-rank $\r = (8,8,2)$. Then, we use the same strategy as in the synthetic experiments to obtain the linear measurements with $m=3\times 10^4$. We set the batch size as $b=0.25m$ for the StoTIHT algorithm. The last two true candle frames and the recovered ones are shown in Figure~\ref{test_StoTIHT_video}. We also present the cost function and relative recovery error in Figure~\ref{test_StoTIHT_video_cost_err}.

\begin{figure}
\begin{minipage}{0.49\linewidth}
\centering
\includegraphics[width=2.8in]{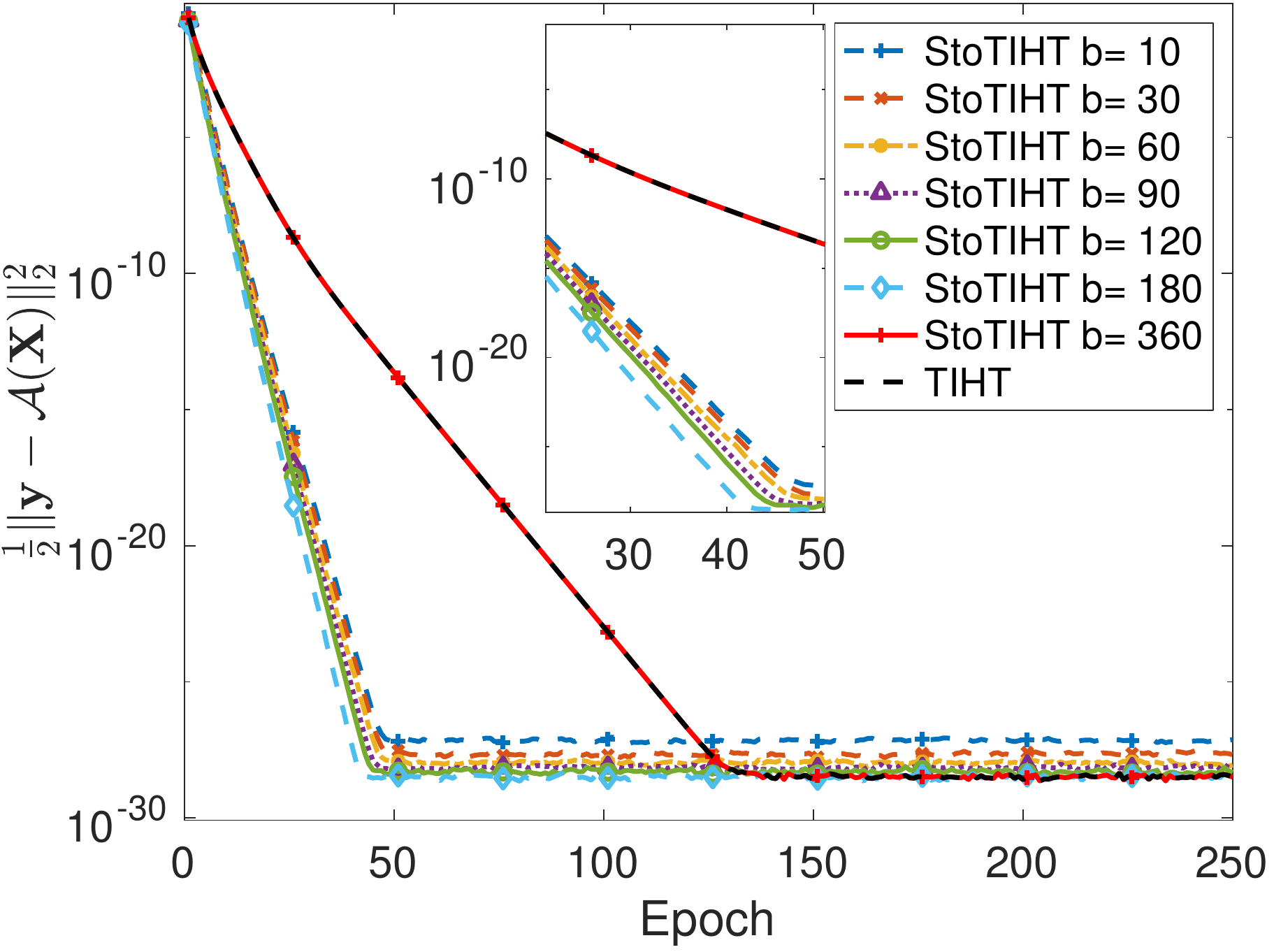}
\centerline{\footnotesize{(a)}}
\end{minipage}
\hfill
\begin{minipage}{0.49\linewidth}
\centering
\includegraphics[width=2.8in]{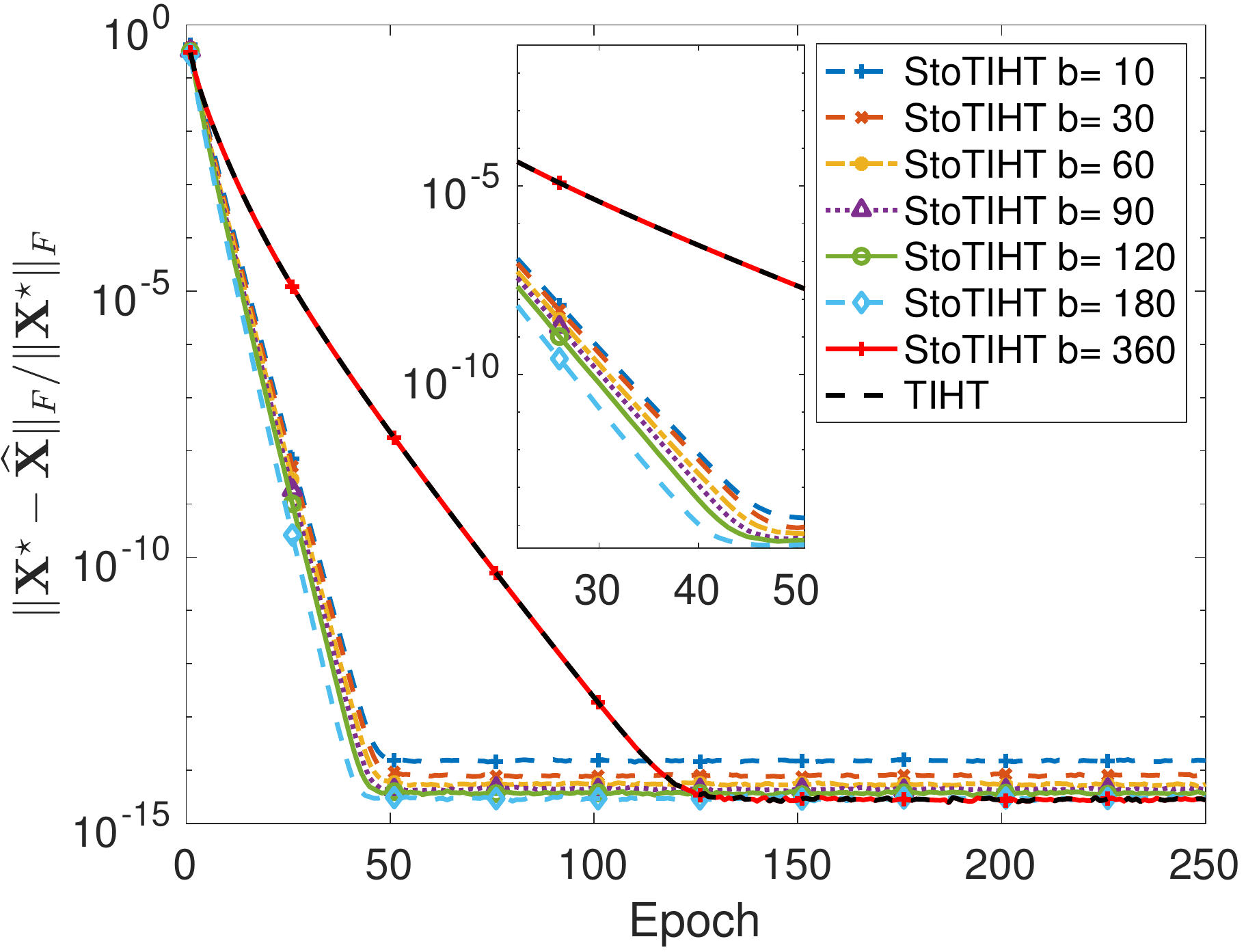}
\centerline{\footnotesize{(b)}}
\end{minipage}
\caption{Low-rank tensor recovery: $n_1 = 5$, $n_2 = 5$, $n_3 = 6$, $m = 360$, $\bm{r} = (1,2,2)$.}
\label{test_StoTIHT_epoch_m100}
\end{figure}

\begin{figure}
\begin{minipage}{0.49\linewidth}
\centering
\includegraphics[width=2.8in]{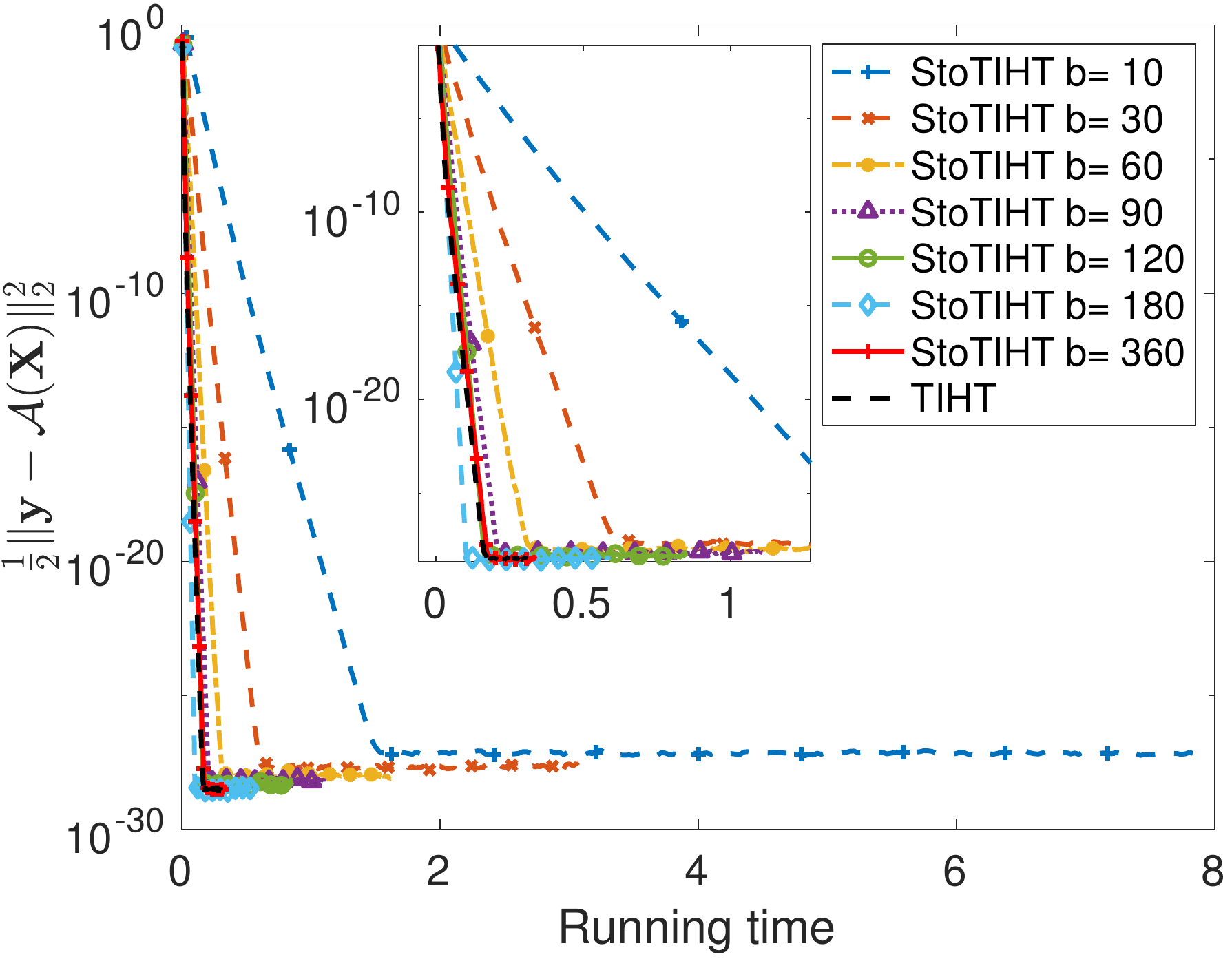}
\centerline{\footnotesize{(a)}}
\end{minipage}
\hfill
\begin{minipage}{0.49\linewidth}
\centering
\includegraphics[width=2.8in]{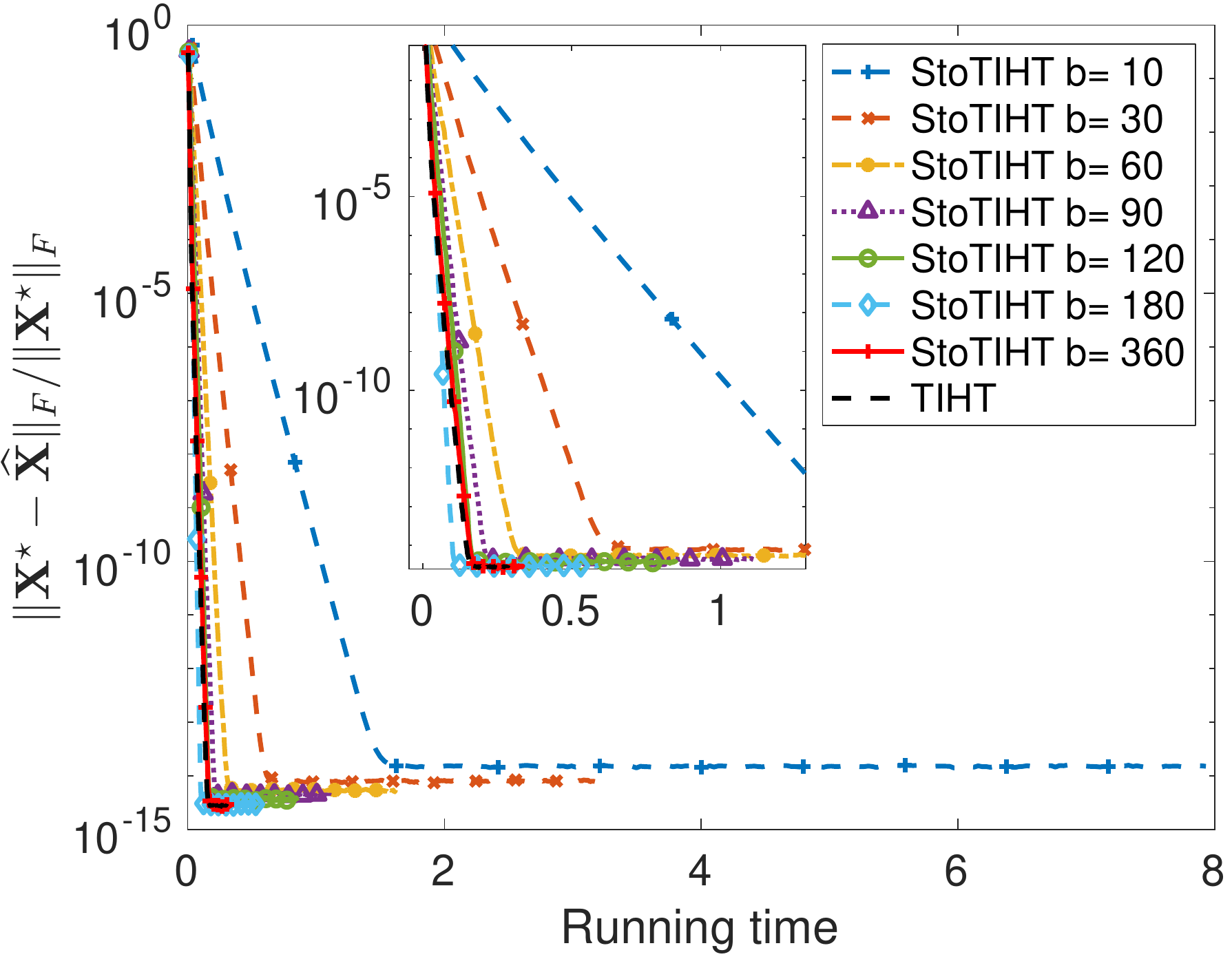}
\centerline{\footnotesize{(b)}}
\end{minipage}
\caption{Low-rank tensor recovery: $n_1 = 5$, $n_2 = 5$, $n_3 = 6$, $m = 360$, $\bm{r} = (1,2,2)$.}
\label{test_StoTIHT_epoch_m100_time}
\end{figure}

\begin{figure}
\begin{minipage}{0.49\linewidth}
\centering
\includegraphics[width=2.8in]{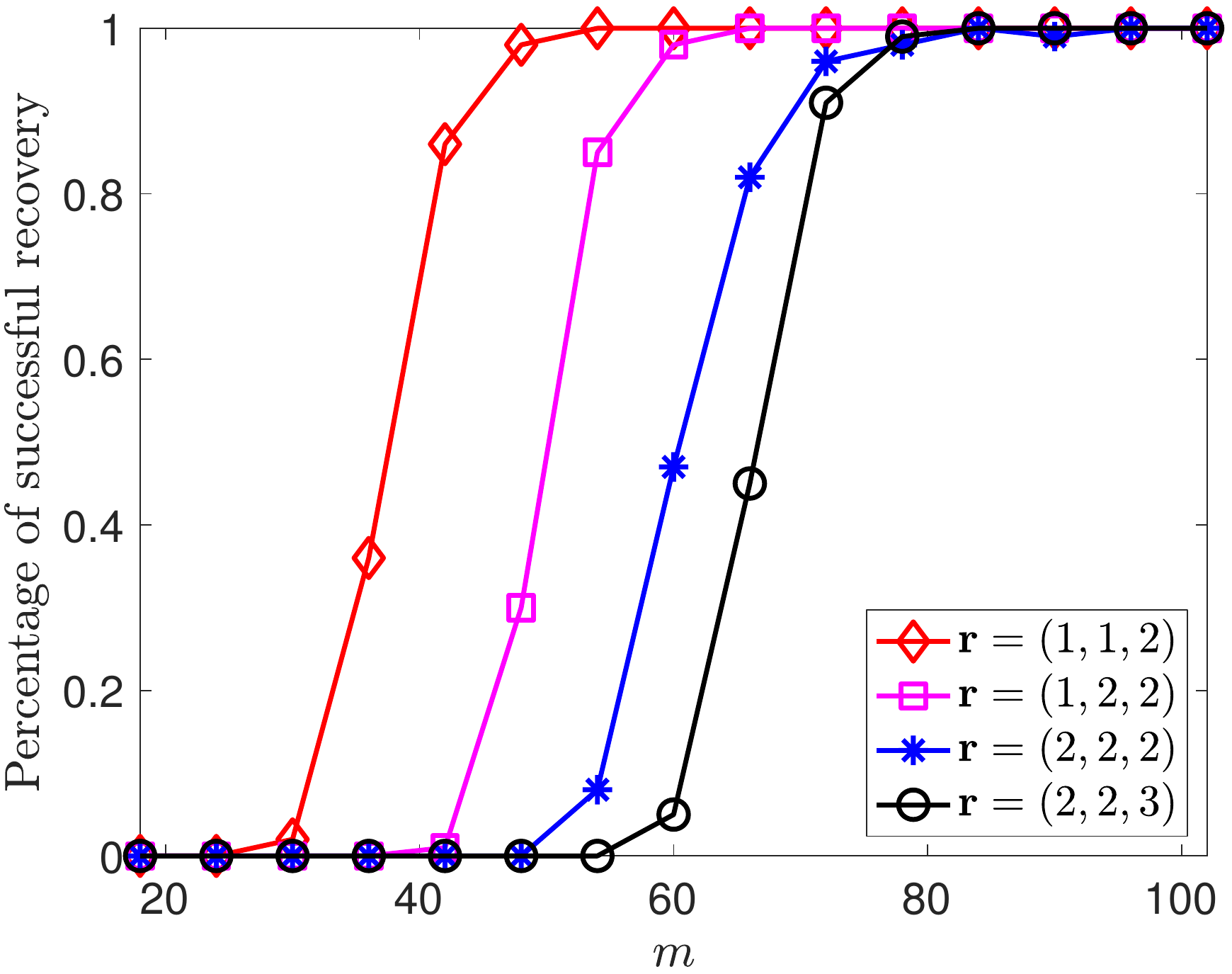}
\centerline{\footnotesize{(a) TIHT}}
\end{minipage}
\hfill
\begin{minipage}{0.49\linewidth}
\centering
\includegraphics[width=2.8in]{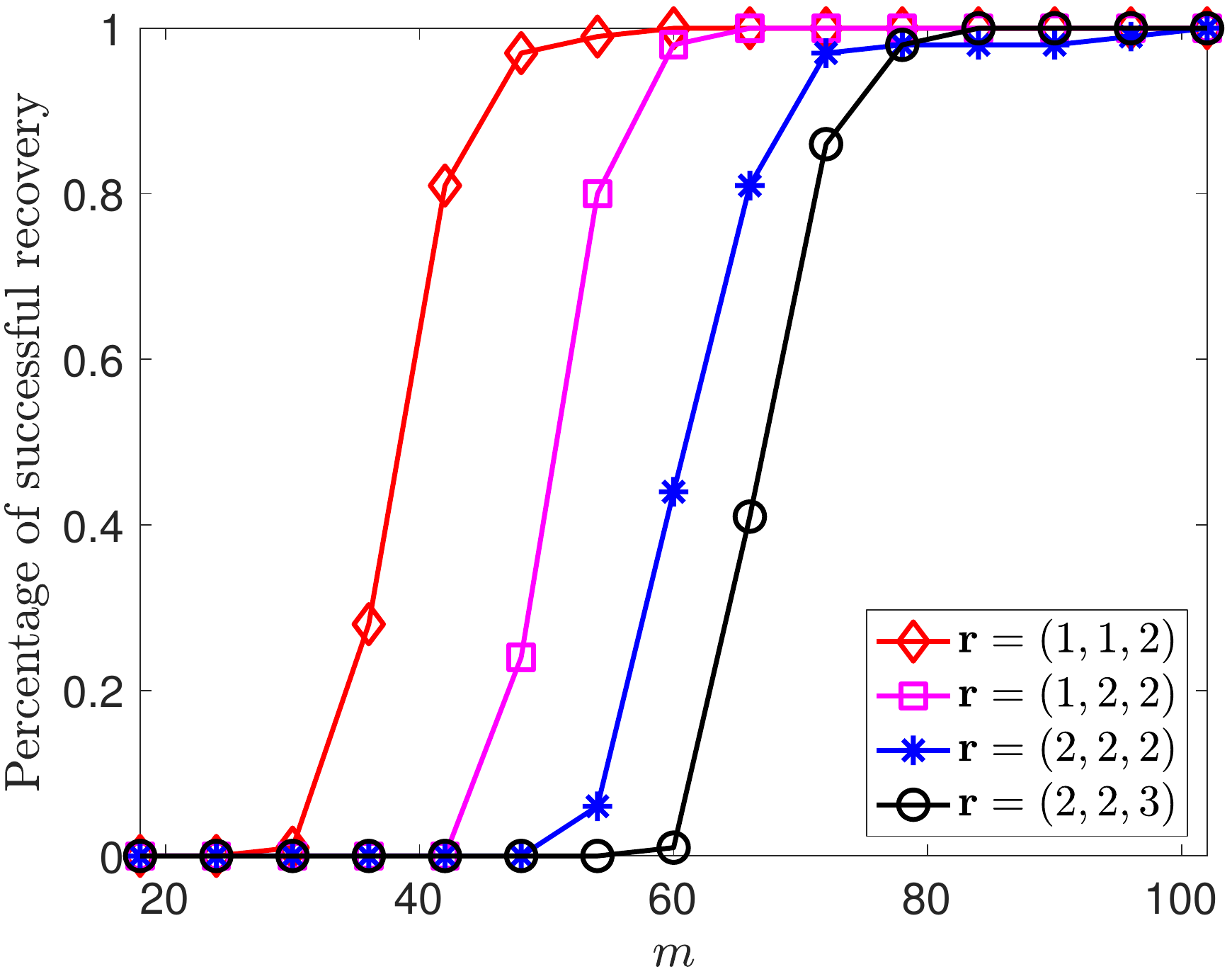}
\centerline{\footnotesize{(b) StoTIHT}}
\end{minipage}
\caption{Low-rank tensor recovery: $n_1 = 5$, $n_2 = 5$, $n_3 = 6$, $b=0.5m$.}
\label{test_StoTIHT_m_r}
\end{figure}

\begin{figure}
\begin{minipage}{0.49\linewidth}
\centering
\includegraphics[width=2.8in]{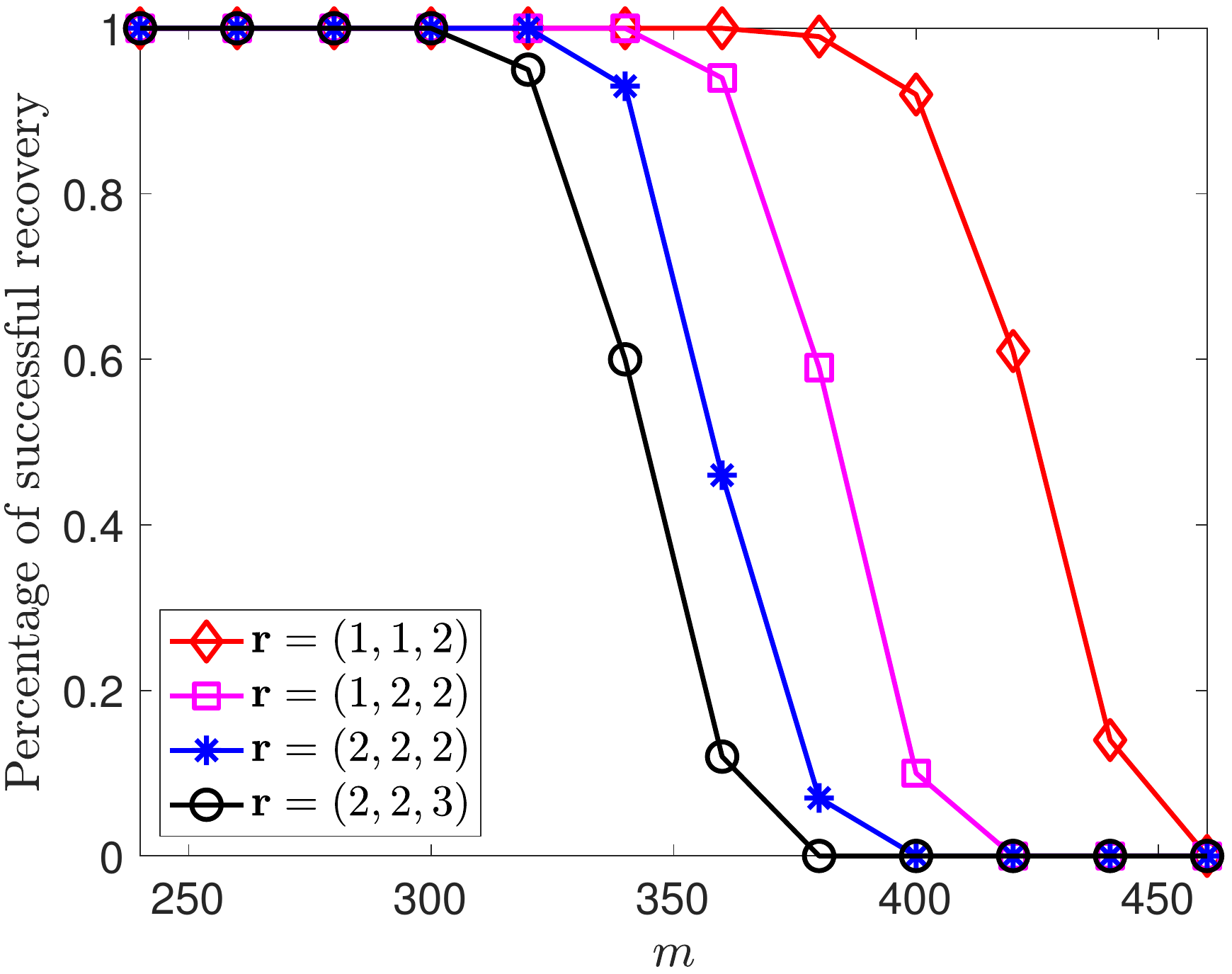}
\centerline{\footnotesize{(a) TIHT}}
\end{minipage}
\hfill
\begin{minipage}{0.49\linewidth}
\centering
\includegraphics[width=2.8in]{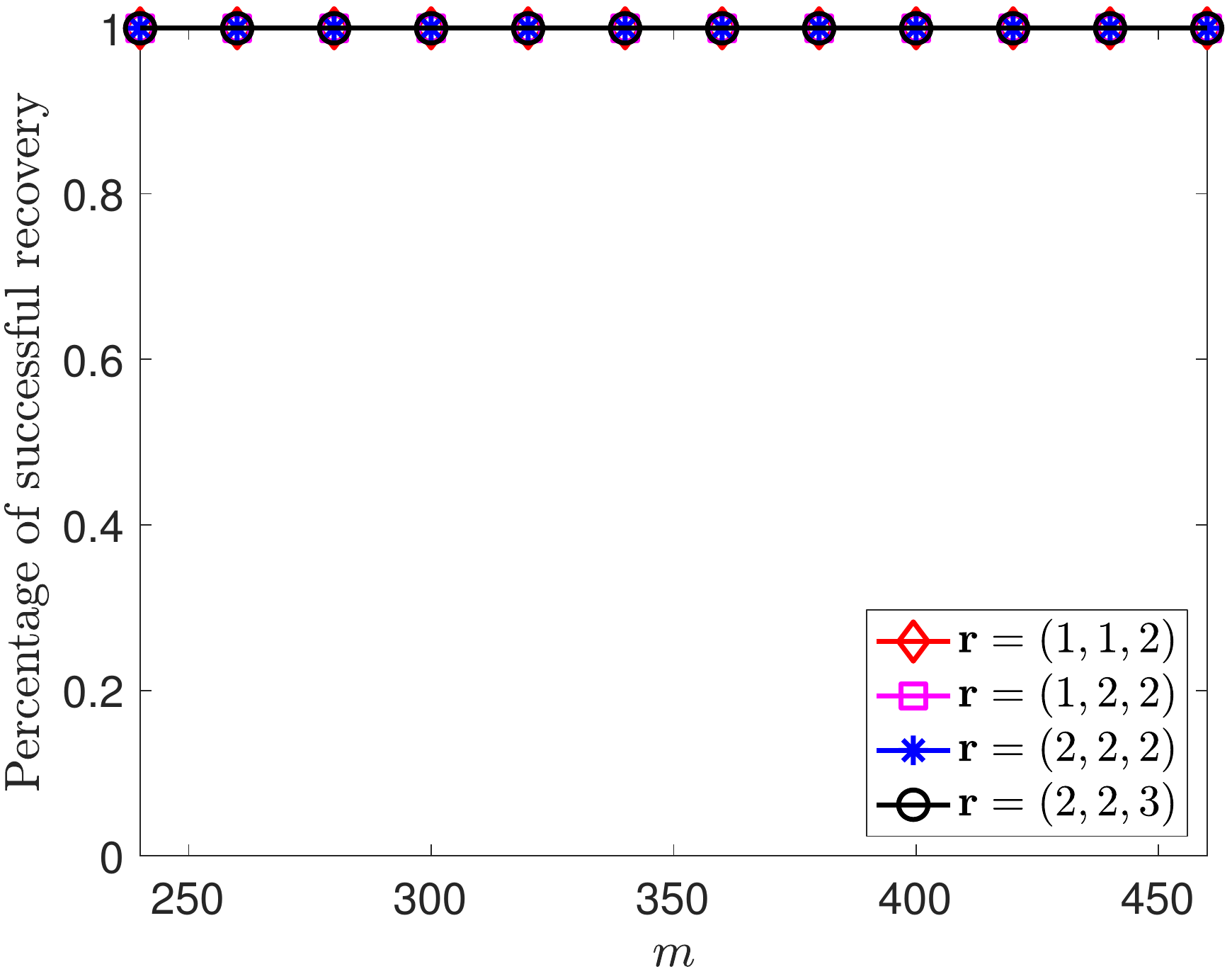}
\centerline{\footnotesize{(b) StoTIHT}}
\end{minipage}
\caption{Low-rank tensor recovery: $n_1 = 5$, $n_2 = 5$, $n_3 = 6$, $b=0.25m$.}
\label{test_StoTIHT_m_r2}
\end{figure}

\begin{figure}
\begin{minipage}{0.49\linewidth}
\centering
\includegraphics[width=2.8in]{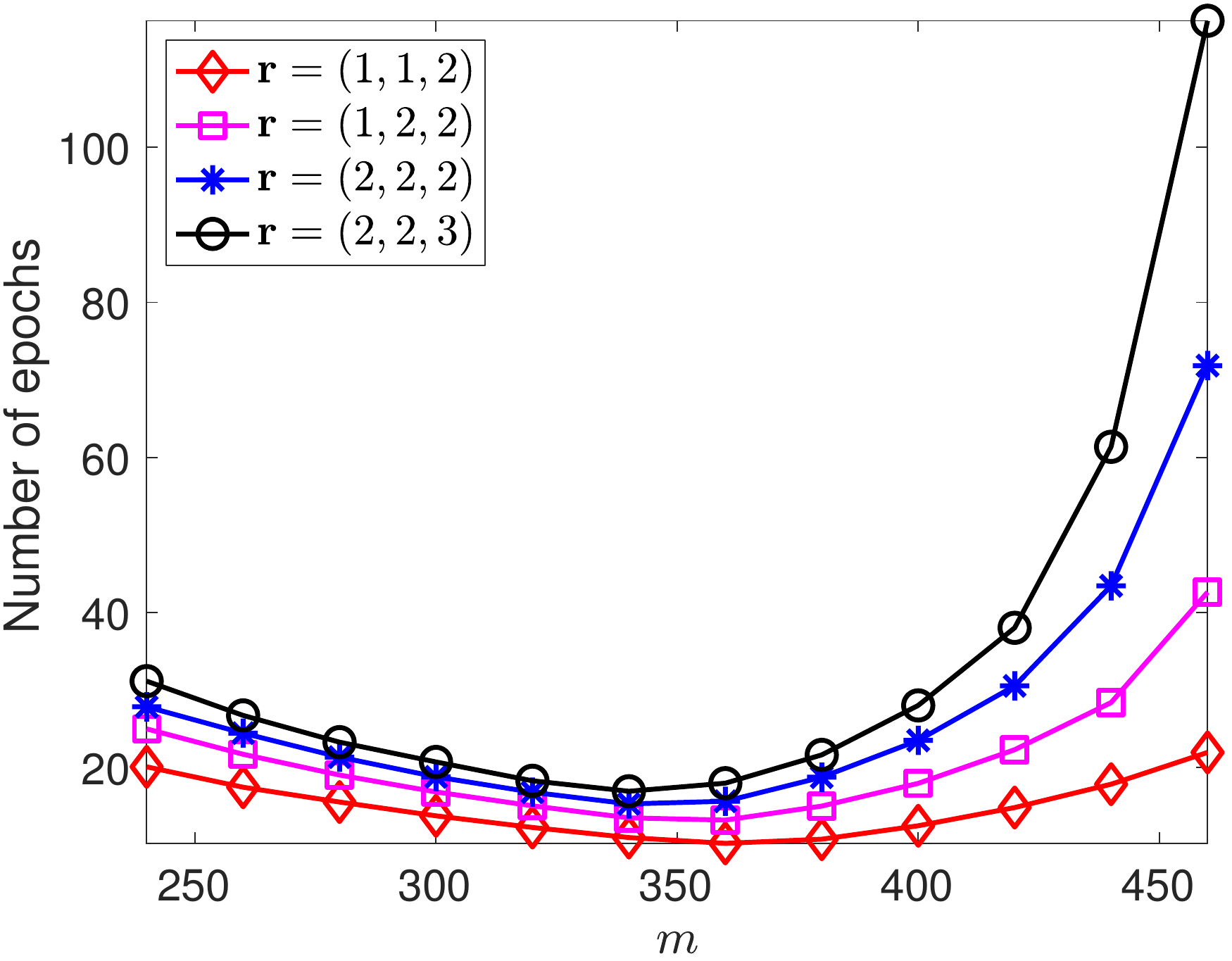}
\centerline{\footnotesize{(a) TIHT}}
\end{minipage}
\hfill
\begin{minipage}{0.49\linewidth}
\centering
\includegraphics[width=2.8in]{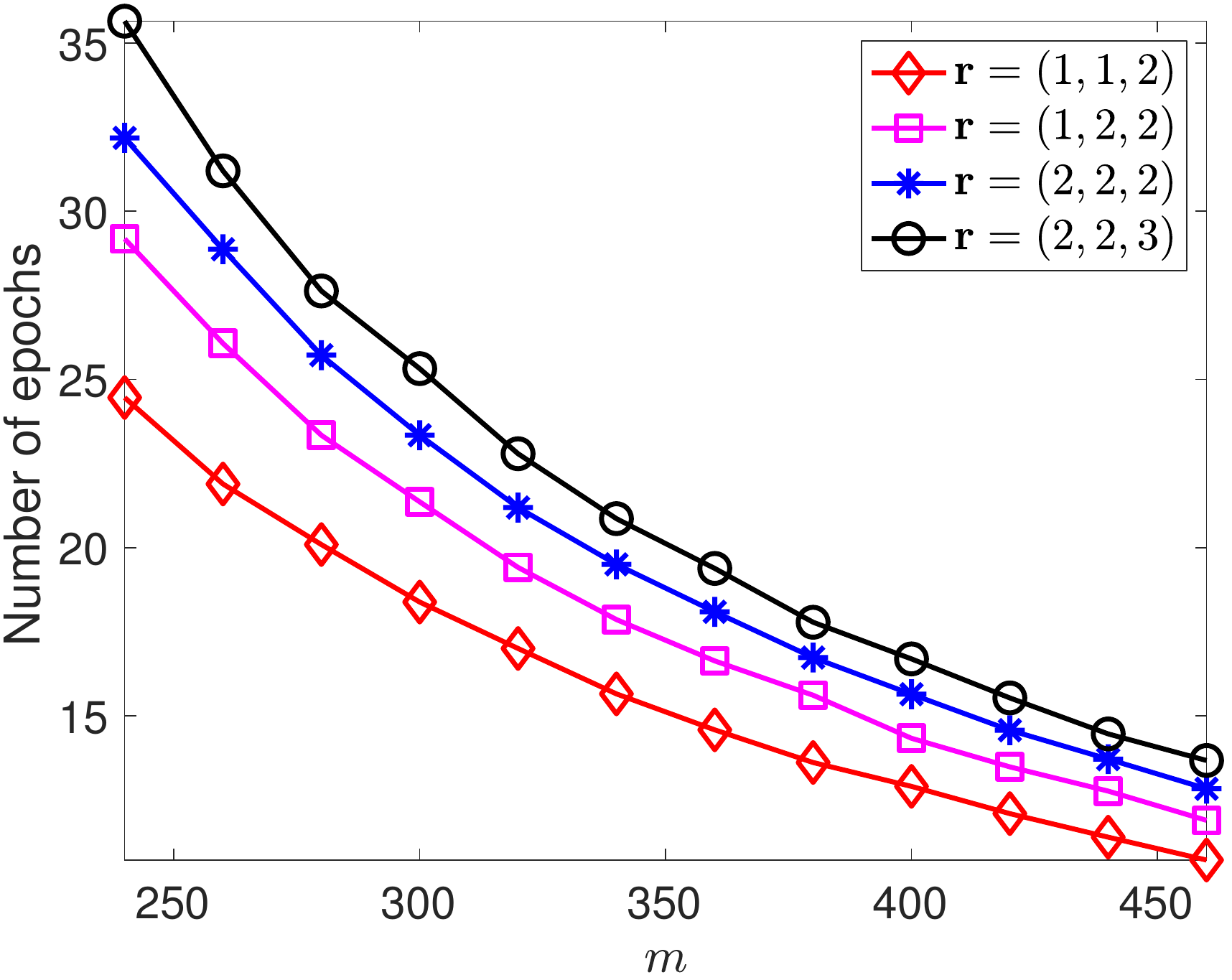}
\centerline{\footnotesize{(b) StoTIHT}}
\end{minipage}
\caption{Low-rank tensor recovery: $n_1 = 5$, $n_2 = 5$, $n_3 = 6$, $b=0.25m$.}
\label{test_StoTIHT_m_r3}
\end{figure}

\begin{figure}
\begin{minipage}{0.31\linewidth}
\centering
\includegraphics[width=1.9in]{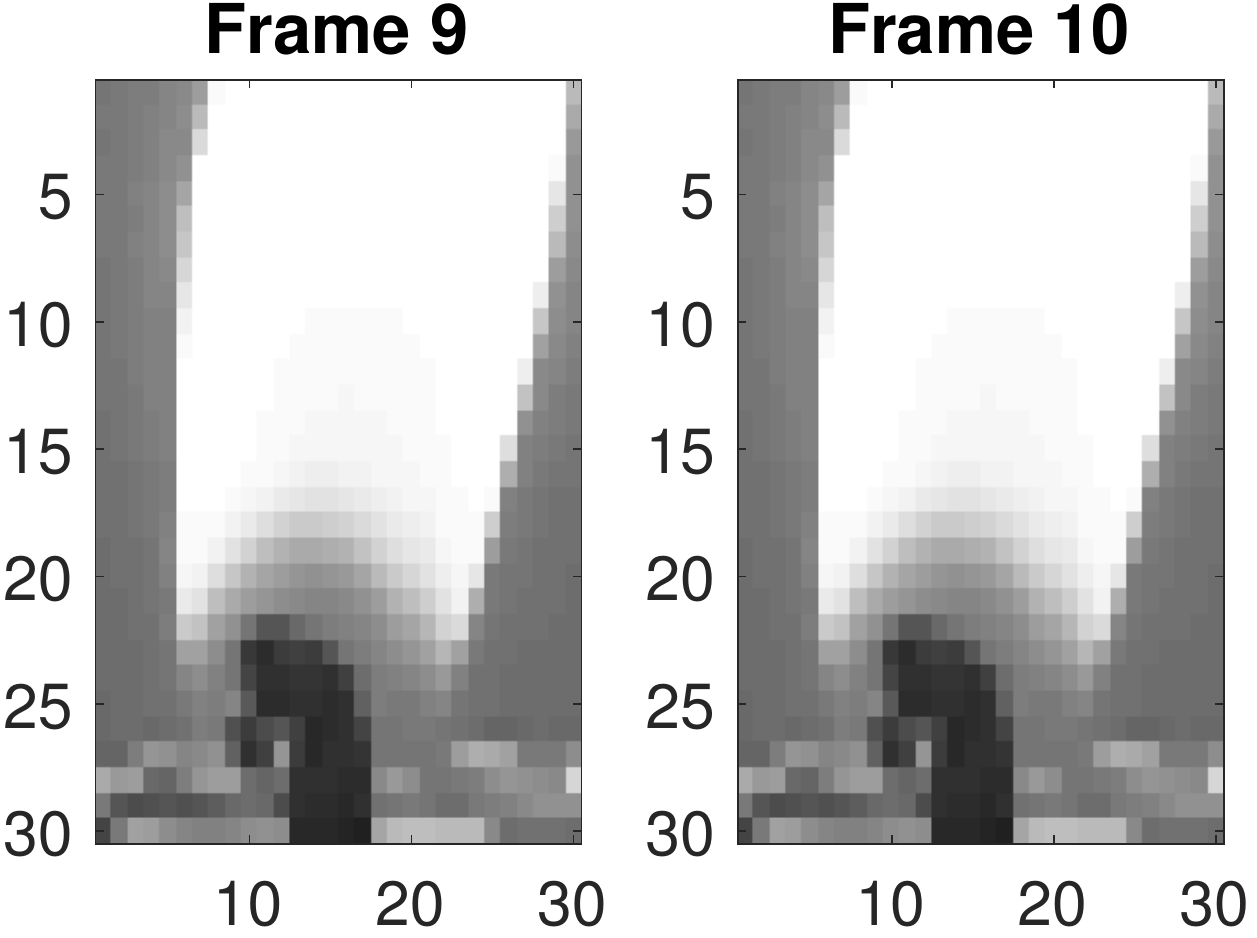}
\centerline{\footnotesize{(a) The original frames}}
\end{minipage}
\hfill
\begin{minipage}{0.31\linewidth}
\centering
\includegraphics[width=1.9in]{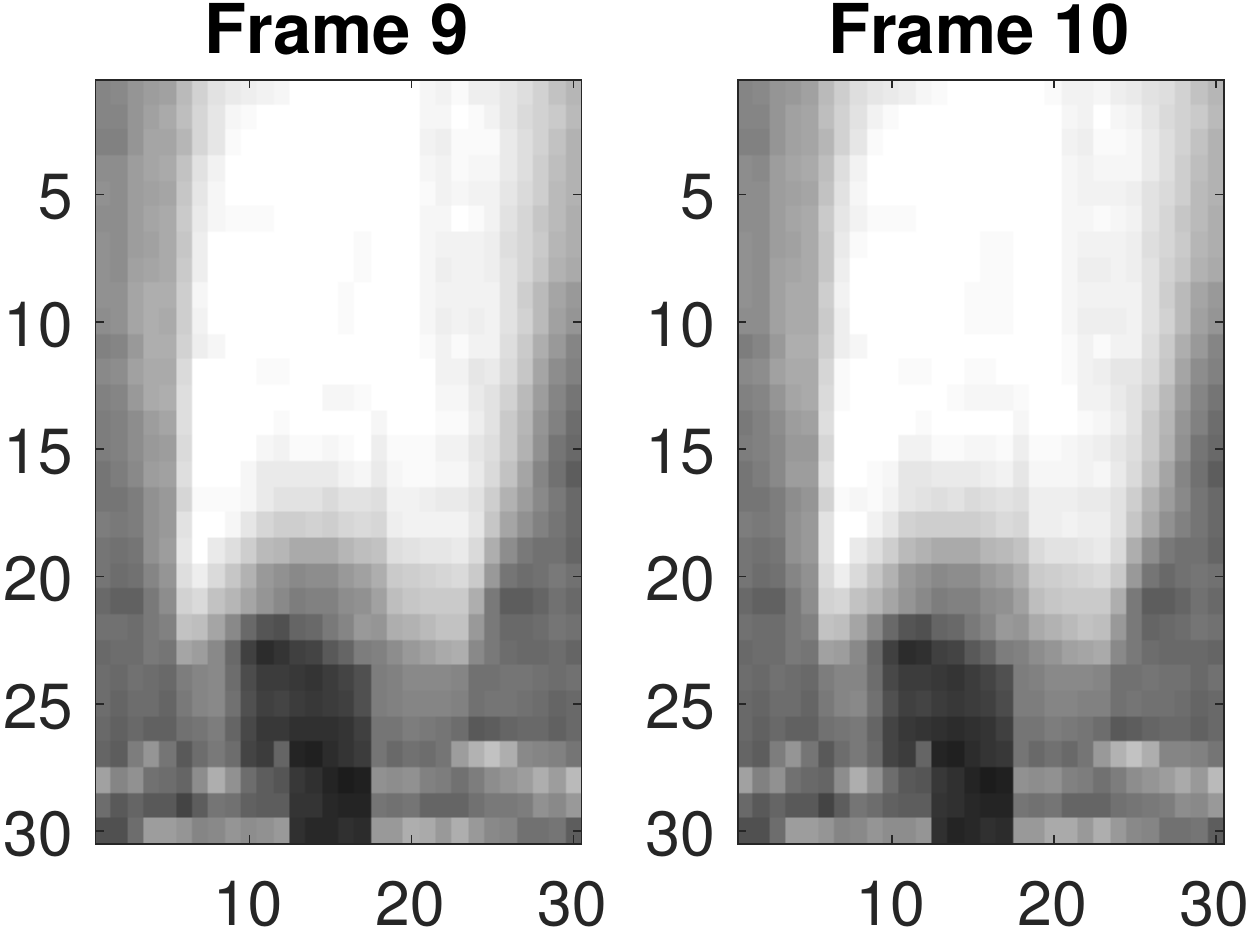}
\centerline{\footnotesize{(b) The recovered frames by TIHT}}
\end{minipage}
\hfill
\begin{minipage}{0.31\linewidth}
\centering
\includegraphics[width=1.9in]{test_StoTIHT_video_estimate2_TIHT}
\centerline{\footnotesize{(c) The recovered frames by StoTIH}}
\end{minipage}
\caption{Candle video recovery: $n_1 = 30$, $n_2 = 30$, $n_3 = 10$, $b=0.25m$.}
\label{test_StoTIHT_video}
\end{figure}


\begin{figure}
\begin{minipage}{0.49\linewidth}
\centering
\includegraphics[width=2.8in]{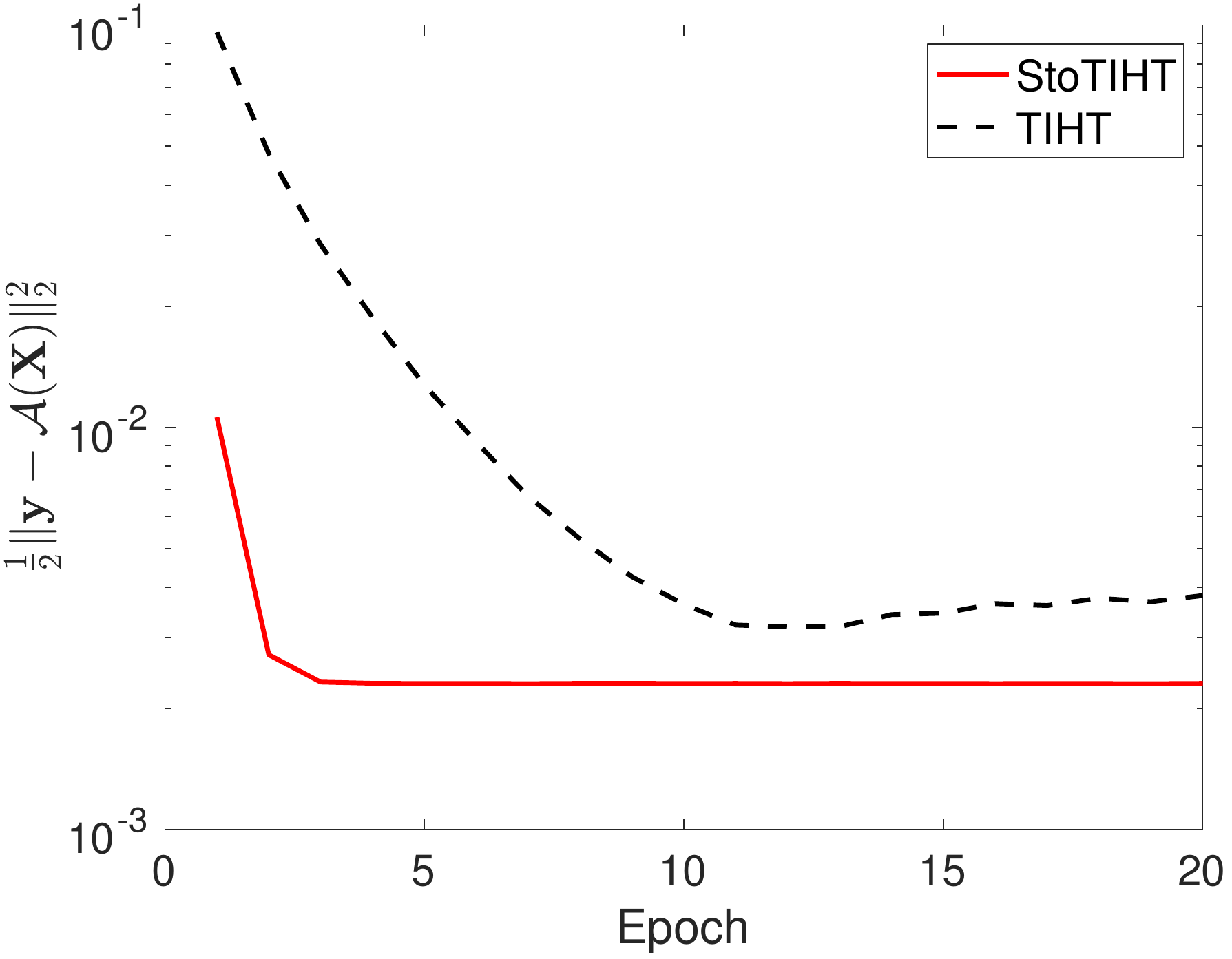}
\centerline{\footnotesize{(a) Cost function}}
\end{minipage}
\hfill
\begin{minipage}{0.49\linewidth}
\centering
\includegraphics[width=2.8in]{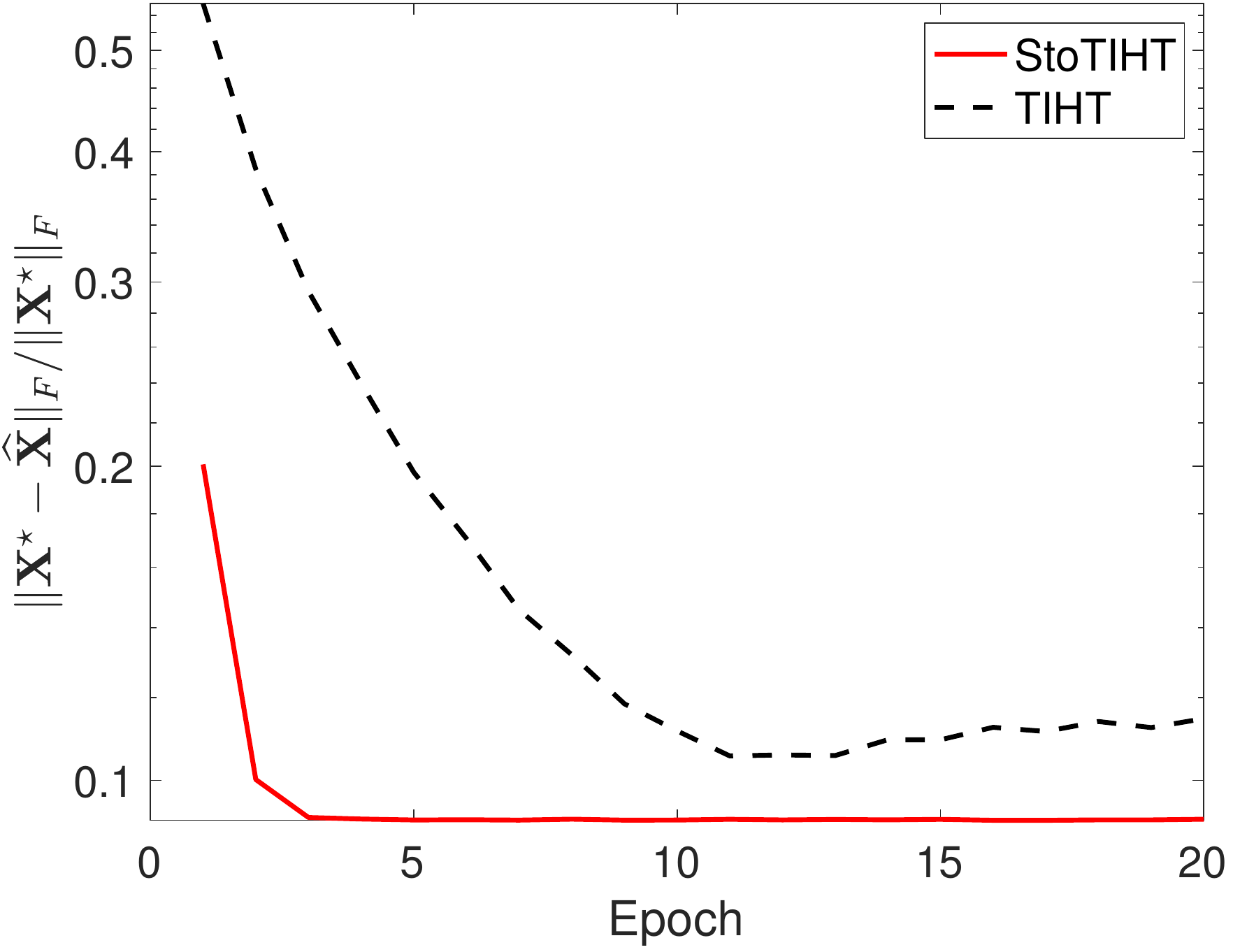}
\centerline{\footnotesize{(b) Relative error}}
\end{minipage}
\caption{Candle video recovery: $n_1 = 30$, $n_2 = 30$, $n_3 = 10$, $b=0.25m$.}
\label{test_StoTIHT_video_cost_err}
\end{figure}

\section{Conclusion}
\label{sec:conc}

In this work, we propose the StoTIHT algorithm by applying a stochastic approach to the TIHT algorithm to address the problem of recovering a low-Tucker-rank tensor from its linear measurements. We have developed theory to guarantee the linear convergence of the proposed algorithm. Our simulation results also indicate that the proposed StoTIHT algorithm significantly outperforms the original TIHT algorithm in the large scale setting. In particular, the proposed StoTIHT algorithm converges much faster and can achieve a lower recovery error when compared with the original TIHT algorithm in the large scale setting.

\section*{Acknowledgments}

This material is based upon work supported by the National Security Agency under Grant No. H98230-19-1-0119, The Lyda Hill Foundation, The McGovern Foundation, and Microsoft Research, while the authors were in residence at the Mathematical Sciences Research Institute in Berkeley, California, during the summer of 2019 as part of the Summer Research for Women in Mathematics (SWiM) program. In addition, Li was supported by the NSF grants CCF-1409258, CCF-1704204, and the DARPA Lagrange Program under ONR/SPAWAR contract N660011824020. Grotheer was supported by the Goucher College Summer Research grant. Needell was supported by NSF CAREER DMS $\#1348721$ and NSF BIGDATA DMS $\#1740325$. Qin was supported by the NSF DMS $\#1941197$.

\bibliographystyle{elsarticle-num-names}
\bibliography{tucker}



\end{document}